\documentclass[11pt, a4paper]{amsart}
\usepackage[utf8]{inputenc}
\usepackage{amsmath}
\usepackage{amssymb}
\usepackage{amsfonts}
\usepackage{amscd}
\usepackage{enumerate}
\usepackage[margin = 0.9in]{geometry}
\usepackage{amsthm}
\usepackage{mathrsfs}
\usepackage{mathtools}
\usepackage{amssymb}
\usepackage[all]{xy}
\usepackage{xcolor}
\usepackage{graphics}
\usepackage{lscape}
\usepackage{array}
\usepackage{microtype}
\usepackage{setspace}
\usepackage{adjustbox}
\usepackage{palatino}
\usepackage{eulervm}
\usepackage{parskip}
\usepackage{marvosym}
\usepackage{tikz-cd} 
\usetikzlibrary{graphs,decorations.pathmorphing,decorations.markings}
\usepackage{stmaryrd} 
\usepackage{upgreek} 
\usepackage{xcolor}
\usepackage{centernot} 
\usepackage[utf8]{inputenc}
\usepackage{comment}
\usepackage[shortlabels]{enumitem}
\usepackage{xy}

\makeatletter
\def\l@subsection{\@tocline{2}{0pt}{2.5pc}{5pc}{}}
\makeatother

\numberwithin{equation}{section}
\newtheorem{claim}{Claim}
\newtheorem*{theorem*}{Theorem}
\newtheorem*{definition*}{Definition}

\newtheorem{theorem}{Theorem}[section]
\newtheorem{lemma}[theorem]{Lemma}
\newtheorem{proposition}[theorem]{Proposition}
\newtheorem{corollary}[theorem]{Corollary}
\newtheorem{conjecture}[theorem]{Conjecture}
\newtheorem{notation}[theorem]{Notation}

\theoremstyle{definition}
\newtheorem{definition}[theorem]{Definition}
\newtheorem{def/prop}[theorem]{Definition/Proposition}

\theoremstyle{remark}
\newtheorem{remark}[theorem]{Remark}
\newtheorem{Construction}[theorem]{Construction}

\usepackage{hyperref}\hypersetup{colorlinks}

\usepackage{color} 

\definecolor{darkred}{rgb}{1,0,0} 
\definecolor{darkgreen}{rgb}{0,1,0}
\definecolor{darkblue}{rgb}{0, 0, 1}
\definecolor{darkpurple}{RGB}{170, 51, 106}

\hypersetup{colorlinks,
linkcolor=darkblue,
filecolor=darkgreen,
urlcolor=darkred,
citecolor=darkpurple}

\usepackage{xspace}

\DeclareMathOperator{\Hom}{Hom}

\DeclareMathOperator{\rig}{rig}
\DeclareMathOperator{\can}{can}

\DeclareMathOperator{\Nm}{Nm}

\DeclareMathOperator{\Prym}{Prym}
\DeclareMathOperator{\id}{1}
\DeclareMathOperator{\Spec}{Spec}

\DeclareMathOperator{\Ob}{Ob}
\DeclareMathOperator{\Fr}{Fr}

\DeclareMathOperator{\Triv}{Triv}

\DeclareMathOperator{\charac}{char}

\DeclareMathOperator{\Tr}{Tr}
\DeclareMathOperator{\Aut}{Aut}
\DeclareMathOperator{\tr}{tr}

\usepackage[backend=biber, maxnames=5,
style=alphabetic,
sorting=nyt]{biblatex}
 \newcommand{\SLn}{\mathrm{SL}_n}
 \newcommand{\GLn}{\mathrm{GL}_n}
\newcommand{\PGLn}{\mathrm{PGL}_n}

\newcommand{\Aa}{\mathcal{A}}
\newcommand{\Bb}{\mathcal{B}}
\newcommand{\Cc}{\mathcal{C}}
\newcommand{\Dd}{\mathcal{D}}

\newcommand{\Gg}{\mathcal{G}}
\newcommand{\Hh}{\mathcal{H}}

\newcommand{\Ll}{\mathcal{L}}
\newcommand{\Mm}{\mathcal{M}}

\newcommand{\Oo}{\mathcal{O}}
\newcommand{\Pp}{\mathcal{P}}

\newcommand{\Tt}{\mathcal{T}}

\newcommand{\Vv}{\mathcal{V}}

\newcommand{\Xx}{\mathcal{X}}

\newcommand{\CC}{\mathbb{C}}

\newcommand{\EE}{\mathbb{E}}
\newcommand{\GG}{\mathbb{G}}

\newcommand{\ZZ}{\mathbb{Z}}
\newcommand{\LL}{\mathbb{L}}
\newcommand{\FF}{\mathbb{F}}

\newcommand{\PP}{\mathbb{P}}
\newcommand{\QQ}{\mathbb{Q}}

\newcommand{\NN}{\mathbb{N}}

\newcommand\Quotient[2]{
\mathchoice
{
\text{\raise1ex\hbox{\thinspace $#1$}\Big{/} \lower1ex\hbox{$#2$} \thinspace}%
}
{
#1\,/\,#2
}
{
#1\,/\,#2
}
{
#1\,/\,#2
}
}

\newcommand\GIT[2]{
\mathchoice
{
\text{\raise1ex\hbox{\thinspace $#1$}\Big{/}\!\!\!\!\Big{/} \lower1ex\hbox{$#2$} \thinspace}%
}
{
#1\,/\,#2
}
{
#1\,/\,#2
}
{
#1\,/\,#2
a       }
}

\newif\ifcomments
\commentstrue

\ifcomments
\newcommand{\authorcomment}[2]{\tikz[baseline=(X.base)]\node [draw=#1,fill=#1!40,semithick,rectangle,inner sep=2pt, rounded corners=3pt] (X) {#2};}

\newcommand{\bl}[1]{\authorcomment{orange}{Basile:} \textcolor{orange}{\textit{#1}}}
\else
\newcommand{\authorcomment}[2]{}

\newcommand{\bl}[1]{}
\fi

\begin{document}

\author{Elsa Maneval}
\title{Non-archimedean topological mirror symmetry for $\mathrm{SL}_n$ and $\mathrm{PGL}_n$ Higgs bundles}

\begin{abstract}
 The Hausel--Thaddeus conjectures concern topological mirror symmetry between moduli spaces of $\mathrm{SL}_n$ and $\mathrm{PGL}_n$ Higgs bundles on a curve. A non-archimedean approach was introduced by Groechenig, Wyss and Ziegler, proving the conjecture for coprime rank and degree. This article is concerned with its generalisation to the non-coprime case. We treat both the classical ($D=K$) and meromorphic ($D>K$) settings. We prove an equality of $p$-adic volumes twisted by gerbes between moduli spaces of $\mathrm{SL}_n$ and $\PGLn$ Higgs bundles of arbitrary degree. In the meromorphic case, building on results of Maulik and Shen, we show that these twisted $p$-adic volumes are related to intersection cohomology. We also conjecture a connection between these $p$-adic volumes and $BPS$ cohomology.
\end{abstract}

 \maketitle



\section{Introduction}



\subsection{Mirror symmetry for \texorpdfstring{$\SLn$}{SL(n)} and \texorpdfstring{$\PGLn$}{PGL(n)} Higgs bundles}

For a fixed rank $n$ and arbitrary degrees $d$ and $e$, we consider the dual Hitchin system of moduli of semi-stable (twisted) $\SLn$ and $\PGLn$-Higgs bundles of degree $d$ and $e$ over a curve $\Cc$. 

    \hfil
    \begin{tikzcd}
        M_{\SLn}^d \arrow[dr, swap, "h_{\SLn}"] & & M_{\PGLn}^e \arrow[dl, "h_{\PGLn}"] \\
        & \Aa = \underset{{i=2}}{\overset{n}{\bigoplus}} H^0(\Cc, K_\Cc^i)&
    \end{tikzcd}


Hausel and Thaddeus \cite{HT} showed that this dual Hitchin fibration satisfies a Strominger-Yau-Zaslow-type duality \cite{SYZ} and, in the coprime case, a topological mirror symmetry in rank $2$ and $3$. This symmetry is expressed in terms of stringy Hodge numbers twisted by a gerbe $\alpha$ :
$$ h^{p,q}(M_{\SLn}^d) = h^{p,q}_{\textnormal{st}, \alpha^d} (M_{\PGLn}^e).$$

The topological mirror symmetry in the coprime case was then proven in \cite{GWZ} and \cite{MSendoscopicHTcoprime} in arbitrary rank. Both proofs work for a larger class of moduli spaces of Higgs bundles, namely meromorphic Higgs bundles. 

The proof in \cite{GWZ} uses $p$-adic Hodge theory and Chebotarev density theorem to reformulate the equality of Hodge numbers in terms of point counting over finite fields. Then, these counts are computed by so-called $p$-adic integrals. The topological mirror symmetry is then deduced from an equality of $p$-adic integrals, which was proven using, among others, the $SYZ$-type symmetry.  

As untwisted $\SLn$-bundles are always of degree $0$, the non-coprime case is also of interest, but harder to understand due to the existence of strictly semi-stable objects. A topological mirror symmetry was conjectured in \cite{hauselsurvey} in the non-coprime case. For the meromorphic case, Maulik and Shen proved it in terms of intersection cohomology \cite{MSnoncoprime}. Mauri also obtained a topological mirror symmetry for intersection cohomology of classical rank $2$ Higgs bundles \cite{Mauri2021}. Recently, a version of these conjectures was formulated in terms of $BPS$ cohomology in \cite{Davison}.

 
 The purpose of this article is to establish certain equalities of $p$-adic integrals (\autoref{mainthm}) that we call non-archimedean topological mirror symmetry, generalising the key intermediate step in \cite{GWZ} to arbitrary rank and degree. 



\subsection{Non-archimedean topological mirror symmetry}
 Our base scheme is $S = \Spec \Oo_F$ where $\Oo_F$ is the ring of integers of a finite extension $F$ of $\QQ_p$ for a prime $p$. We fix a curve $\Cc \longrightarrow S$ and a divisor $D$ on $\Cc$. 
 
 Let $d, e \in \ZZ$ denote degrees on $\SLn$ and $\PGLn$ sides respectively. For $L \in \Pp ic^d(\Cc)$, we consider the coarse moduli space $M^{L}_{\SLn}$ of Higgs bundles over $\Cc$ of determinant $L$, with traceless Higgs field valued in $D$. Moreover, $M_{\PGLn}^e$ is the coarse moduli space of $\PGLn$-Higgs bundles of degree $e$. In \autoref{slngerbe} and \ref{pglngerbe} we construct gerbes $\alpha$ and $\alpha_N$ for $N \in \Pp ic^1(\Cc)$  on the stacks of $\SLn$ and $\PGLn$ Higgs bundles respectively. 

For a generic $a \in \Aa(S)$ corresponding to a spectral curve $\tilde{\Cc} \xlongrightarrow{\pi} \Cc$, the BNR correspondence \cite{BNR89} states that $h_{\SLn}^{-1}(a) \subset \Mm^L_{\SLn}$ is identified with $\Pp ic_{\tilde{\Cc}/S}^{d'}$, where $d' = d -\deg D \cdot \frac{n(n-1)}{2}$. Similarly, we write $e' = e - \deg D \cdot  \frac{n(n-1)}{2}$.

In \autoref{measure}, we introduce measures on certain subspaces $M^L_{\SLn} (\Oo_F)^\sharp \subset M^L_{\SLn}(\Oo_F)$ and $M^e_{\PGLn} (\Oo_F)^\sharp \subset M^e_{\PGLn}(\Oo_F)$. We define integrable functions $f_{\alpha} \colon M(\Oo_F)^\sharp \longrightarrow \CC$ associated to gerbes in \autoref{sub gerbe}. Our main result is the following equality of $p$-adic integrals.

\begin{theorem}[Non-archimedean topological mirror symmetry, \autoref{mainthm}] \label{thm1}
 Let $D \sim K_\Cc$ or ${\deg D > 2g-2}$. We have an equality 
     $$ \int_{M_{\SLn}^L(\Oo_F)^\sharp} f_{\alpha}^{e'} \ \mu_{\can, \SLn} = \int_{M_{\PGLn}^e(\Oo_F)^\sharp} f_{\alpha_{N}}^{d'} \cdot f_L \ \mu_{\can, \PGLn} $$
     where $f_L$ is a function which depends on $[L \otimes \Dd \otimes N^{-d'}] \in \Pp ic^0_{\Cc/S}(S)\big/\Pp ic^0_{\Cc/S}(S)^{\times n}$.  
\end{theorem}

This non-archimedean topological mirror symmetry is relating $M_{\SLn}^{L}$ and $M_{\PGLn}^{e}$ for arbitrary degrees $d$ and $e$, without assuming $(n,d)=(n,e)$ as in \cite[Theorem 0.2 (b)]{MSnoncoprime}. It suggests the existence of a more general topological mirror symmetry statement for $\SLn$ and $\PGLn$, using twists by gerbes. For that we need to interpret the $p-$adic integrals in terms of cohomological invariants. 

\subsection{Interpretations of \texorpdfstring{$p$}{p}-adic integrals} In general, it still remains to interpret $p$-adic integrals in terms of cohomological invariants of the moduli spaces over $\CC$. We can do it in the coprime case, following \cite{GWZ}. 

Let the base scheme be $S = \Spec \CC$. As before, we fix a smooth projective curve $\Cc$ with line bundles $D$ and $N$, $\deg D > 2g -2 $ or $D=K_\Cc$ and $\deg N = 1$. 

The finite abelian group $\Gamma = \Pp ic^0(\Cc)[n]$ acts on $M^L_{\SLn}$ by tensorisation. It induces an isotypical decomposition in cohomology. Moreover, for any $L_e \in \Pp ic^e(\Cc)$, $M^e_{\PGLn} = M^{L_e}_{\SLn}/\Gamma$. 

\begin{corollary}[Corollary \ref{corTMScoprime}]
Let $d, e \in \ZZ$ such that $(d,n)=(e,n)=1$. Let $L \in \Pp ic^d(\Cc)$. The following equality of stringy  twisted $E$-polynomials (Definition \ref{defi stringy}) holds :
$$ E(M_{\SLn}^{L}) = E_{\textnormal{st}, \alpha_N^{d'}}(M_{\PGLn}^{e}). $$
Moreover, for any $\gamma \in \Gamma = \Pp ic^0(\Cc)[n]$ and $L_e \in \Pp ic^e(\Cc)$, there is a refinement : 
$$ E(M_{\SLn}^{L})_{\kappa} = E(M_{\SLn}^{L_e, \gamma})_{\kappa^{d'}} $$
where $M_{\SLn}^{L_e, \gamma}$ denote the $\gamma$-fixed locus in $M_{\SLn}^{L_e}$ and $\kappa = \omega(\gamma)$ as in Remark \ref{rmk weil pairing}.
\end{corollary}

It extends \cite[Theorem 7.21]{GWZ} where the degree of $D$ is assumed to be even. It also allowed us to detect an inaccuracy in \cite[Theorem 3.2]{MSendoscopicHTcoprime} when $d' \not\equiv d \mod n$, see Remark \ref{inaccuracy}. 

In the non-coprime case, the moduli space of $\SLn$-Higgs bundles is singular, and there is not yet a general result relating $p$-adic integrals to cohomological invariants. We can nonetheless use the existing mirror symmetry results \cite[Theorem 0.2]{MSnoncoprime} to obtain the following cohomological interpretation of $p$-adic integrals. The right hand side is an alternated trace of Frobenius on $\ell$-adic cohomology, see Notation \ref{notation frob}.

\begin{proposition}[Proposition \ref{prop link IH}]
Let $S = \Spec \Oo_F$. Let $n$ be an odd prime and $d \in \ZZ$. Let ${L= N^{\otimes d} \in \Pp ic^d(\Cc)}$. Let $k_F \simeq \FF_q$ be the residue field of $\Oo_F$. Then, 
    $$ \int_{M_{\SLn}^L(\Oo_F)^\sharp} f_{\alpha} \ \mu_{\can} = q^{- \dim M_{\SLn}^L} \tr ( \Fr \ | \ IH_c(M_{\SLn, k_F}^L)).$$
\end{proposition}

\begin{remark}
    If such correspondence between $p-$adic integrals and intersection cohomology was established for any $n$ and $D > K_\Cc$ then \autoref{thm1} would imply an identity of intersection $E$-polynomials :
    $$ IE(M^L_{\SLn, \CC}) = E_{\textnormal{st}, \alpha^{d'}}(M^1_{\PGLn, \CC}),$$
    for any $d \in \ZZ$ and $L \in \Pp ic^d(\Cc)$, with $d' = d -\deg D \cdot \frac{n(n-1)}{2}$.
\end{remark}

 Finally, the recent Hausel--Thaddeus-type conjecture \cite[Conjecture 10.3.25]{Davison} suggests an interpretation of these $p$-adic integrals in terms of $BPS$ cohomology. It is speculative, since the sheaf $\phi_{BPS}$ is not yet constructed as an $\ell$-adic sheaf. We continue to work over $S = \Spec \Oo_F$.
 
\begin{conjecture}[Conjecture \ref{conjecture}]  Let $D=K_\Cc$. Let $d \in \ZZ$ and $L \in \Pp ic^d(\Cc)$. Then, 
    $$ \int_{M_{\SLn}^{L}(\Oo_F)^\sharp} f_\alpha \ \mu_{\can} =  q^{- \dim M_{\SLn}^{L}} \tr( \Fr \ | \  H^* (\Ll M_{\SLn, k_F}^{L}, \phi_{BPS})_1). $$ 
\end{conjecture}

\subsection{About the proof} The proof of our main theorem relies on a generalisation to arbitrary fields of \cite[Propositions 3.2 and 3.6]{HT}, called $SYZ$-type mirror symmetries. In \autoref{pryms}, we introduce the notation for abelian varieties $\Pp rym^{0}_{\PGLn}$ and $\Pp rym_{\SLn}^{0}$ acting on generic fibres of Hitchin fibrations. It is already established that they are dual abelian varieties. For the gerbes $\alpha$ and $\alpha_N$ we define a torsor of trivialisations $Triv^0$ in Definition \ref{triv0}.

As in \cite{HT}, we express the torsor structure of a Hitchin fibre in terms of the torsor of trivialisations of the gerbe on the dual fibre. We denote $\Aa$ the Hitchin basis and $\Aa^{\textnormal{sm}}$ its regular locus, see \autoref{pryms}. 

\begin{proposition}[$SYZ$-type symmetries, Propositions \ref{firstSYZ} and \ref{secondSYZ}]
Let $S = \Spec k$ for a field $k$ in which $n$ is invertible and $\mu_{n, k}$ is constant. Let $d, e \in \ZZ$ and $L \in \Pp ic^d_{\Cc/S}(S)$. For $a \in \Aa^{\textnormal{sm}}(k)$, consider $h^{-1}_{\PGLn} (a) \subseteq M^e_{\PGLn}$ and $h^{-1}_{\SLn} (a)  \subseteq M^L_{\SLn} $. 

\begin{enumerate}
    \item There is an isomorphism of $\Pp rym^{0}_{\PGLn}$-torsors 
$$ \Triv^0 (\alpha^{e'} \ | \ h^{-1}_{\SLn} (a))  \xlongrightarrow{\sim} h^{-1}_{\PGLn} (a).  $$ 
    
    \item Let $L_0 = L \otimes \Dd \otimes N^{-d'} \in \Pp ic^0_{\Cc/S}(S)$. There is an isomorphism of $\Pp rym_{\SLn}^{0}$-torsors 
$$ \Triv^0(\alpha_N^{d'} \ | \ h^{-1}_{\PGLn} (a)) \times^{\Pp rym_{\SLn}^{0}} Nm_{\tilde{\Cc}/\Cc}^{-1}(L_0) \xlongrightarrow{\sim} h^{-1}_{\SLn} (a),  $$ 
where $ \times^{\Pp rym_{\SLn}^{0}}$ denote the product of $\Pp rym_{\SLn}^{0}$-torsors. 
\end{enumerate} 
\end{proposition}

\subsection{Organisation of the article.} In \autoref{prelim} we explain the construction of the gerbes and some background for the proofs of $SYZ$-type dualities. 

The moduli stack of $\SLn$-Higgs bundles can be seen as a $\mu_n$-gerbe $\alpha_n$ over its rigidification. However, the $SYZ$ duality is better described in terms of the induced $n$-torsion $\GG_m$-gerbe $\alpha$, which is also naturally a moduli stack of Higgs bundles.

For the $\PGLn$-side, the moduli stack is constructed as a quotient by the finite group $\Gamma = \Pp ic^0_{\Cc/S}[n] \simeq (\ZZ/n\ZZ)^{2g}$ of the $\SLn$ stack. The action of $\Pp ic^0_{\Cc/S}$ by tensorisation on $\GLn$-bundles restricts to an action of $\Gamma$ on $\SLn$-bundles. Often one defines the $\Gamma$-action directly on the rigidified stack, where scalar automorphisms do not appear. However, rigidifying line bundles in $\Gamma$, we get a $\Gamma$-action at the level of the non-rigidified stack. It gives the gerbe $\alpha$ a structure of $\Gamma$-equivariant $\GG_m$-gerbe, hence makes it a gerbe over the $\PGLn$ moduli space. This construction requires the choice of a degree $1$ line bundle $N$ over $\Cc$, which replaces the base point $c \in \Cc$ in \cite{HT}. We explain in Appendix \ref{appendix} the precise relation to their setup. To characterize the trivialisations of the $\Gamma$-equivariant gerbe, we study a relative group scheme $\Gamma_N \longrightarrow \Cc$, which acts on Higgs bundles corresponding to $\Gamma$-equivariant sections of the gerbe $\alpha$.

In \autoref{section SYZ}, we generalise the $SYZ$-type duality of \cite{HT} to arbitrary fields and arbitrary determinant line bundle $L$. 

The main theorem is proven in \autoref{sectionnaTMS}. The proof uses the $p$-adic measure constructed in \cite{COW} and the arguments of \cite{GWZ}. Namely, we interpret gerbe functions in terms of Tate duality pairings when integrating against abelian varieties. 

Finally, \autoref{section corollaries} is devoted to completing the proof of topological mirror symmetry in the coprime case with odd degree $D$ and to the interpretation of $p$-adic integrals in terms of intersection and $BPS$ cohomologies.

\subsection{Acknowledgements.}
I would like to thank Dimitri Wyss for suggesting this project and for his generous advice throughout. I am grateful to Junliang Shen and Mirko Mauri for valuable correspondences. Discussions with Eric Yen-Yo Chen, Archi Kaushik, Sebastian Schlegel Mejia and Tanguy Vernet also played an important role. 

Financial support was provided by grant \# 218340 from the Swiss National
Science Foundation (FNS/SNF). 




\section{Preliminaries} \label{prelim}

In this article we use $S$ to denote a base scheme. It will always be $\Spec k$ for a field $k$, $\Spec \Oo_F$ for a local ring with finite residue field $k$ or $\Spec R$ for a finitely generated $\ZZ$-algebra $R \subset \CC$. 

We fix an arbitrary rank $n \in \NN$. We always assume that $\charac \ k$ does not divide $n$ and that $\mu_{n, S}$ is a constant group scheme. 


Let $\Cc \longrightarrow S$ be a relative smooth proper curve over $S$ of genus $g \geq 2$. We further assume that the group scheme $\Gamma = \Pp ic^0_{\Cc/S} [n]$  is constant over $S$.

\subsection{Notations for Picard varieties} \label{subsection notations picard}

In this article, we consider smooth integral curves $\Tilde{\Cc} \longrightarrow S$ over $S = \Spec \Oo_F$ and $S = \Spec k$. In all cases, $\Tilde{\Cc}  \longrightarrow S$ the relative Picard functor $Pic_{\Tilde{\Cc}/S}$ is represented by a $S$-scheme $\Pp ic_{\Tilde{\Cc}/S} $ \cite[séminaire n°232, Théorème 3.1]{grotFGA}. Moreover, by flatness and irreducibility, the degree of line bundles is well-defined.

Since $\Pp ic_{\Tilde{\Cc}/S}$ represents the relative Picard functor, there is a unique Poincaré invertible sheaf normalized at $(0,0)$  \cite[Ex. 9.4.3]{FGAexplained}. We denote it 
$$ \Pp_{\Tilde{\Cc}} \longrightarrow \Pp ic^0_{\Tilde{\Cc}/S} \times \Pp ic^0(\Pp ic^0_{\Tilde{\Cc}/S} ).$$
 

There is canonical self-duality isomorphism 
\begin{align*}
     \Pp ic^0_{\Tilde{\Cc}/S}  &\xlongrightarrow{\sim} \Pp ic^0(\Pp ic^0_{\Tilde{\Cc}/S}) \\ 
     \Ll &\longmapsto \Oo_{\Pp ic^0_{\Tilde{\Cc}/S}}([\Ll] - [0 ]).
\end{align*}

For the fixed curve $\Cc$, we denote $\Gamma = \Pp ic^0(\Cc) [n]$. Using self-duality, we can consider $\Gamma \subset \Pp ic^0(\Pp ic^0_{\Cc/S})$. 

\begin{notation} \label{notation LL et L}
    For $\gamma \in \Gamma$, we denote $\LL_\gamma$ (resp. $L_\gamma$) the line bundle over $\Pp ic^0_{\Cc/S}$ (resp. $\Cc$) associated to $\gamma$. 
\end{notation}

We pick a trivialising section $s$ of $$[0]_{\Pp ic^0(\Pp ic^0_{\Cc/S})}^*\Pp_\Cc \longrightarrow \Pp ic^0_{\Cc/S}.$$ 

For $\gamma \in \Gamma$, $s$ provides a non-zero rational point $P_\gamma$ in the fibre over $[0]$ of $\LL_\gamma$. 

Since $\Cc$ embeds canonically in $\Pp ic^1_{\Cc/S}$ via the Abel-Jacobi map (denoted AJ), any $N \in \Pp ic^1_{\Cc/S}(S)$ allows us to identify $L_\gamma$ with the pullback of $\LL_\gamma$ along $\Cc \xrightarrow{AJ} \Pp ic^1_{\Cc/S} \xrightarrow{\cdot N^{-1}} \Pp ic^0_{\Cc/S}.$ 

\begin{notation} \label{gN}
    For $N \in \Pp ic^1_{\Cc/S}(S)$, we use the notation $g_N$ for the isomorphism :
\begin{align*} 
    g_{N} \ : \  \Pp ic^1_{\Cc/S}  &\xlongrightarrow{\sim} \Pp ic^0_{\Cc/S}  \\
    L &\longmapsto L \otimes  N^{-1}.
\end{align*}
\end{notation}

Let $S = \Spec k$. In general, for a line bundle over an abelian variety $\LL \longrightarrow A$, there is an associated group scheme $\Gg(\LL)$ called the theta group \cite[§8]{Moonen}, \cite[§5.11]{bookcontainingMilneabelianvarieties}.  For a $S$-scheme $T$, a $T$-section of $\Gg(\LL)$ is given by a pair $(x,\phi)$ with $x \in A(T)$ and $\phi \ : \ t_x^* \LL_T \xlongrightarrow{\sim} \LL_T$ where $t_x$ denotes the translation by $x$. Let $\LL^*$ denote the $\GG_m$-torsor associated to $\LL$. If $\LL \in \Pp ic^0(A)$, the choice of a rational point $P \neq 0$ in the fibre over $[0]$ of $\LL $ provides an isomorphism of schemes $\Gg(\LL) \xlongrightarrow{\sim} \LL^*$, given by $(x, \phi) \longmapsto \phi(P)$. 

Let $\Pp_\Cc \setminus 0$ denote the Poincaré bundle without its zero section. Using the second projection, there is a map $\Pp \setminus 0 \longrightarrow \Pp ic^0(\Pp ic^0_{\Cc/S})$. The fibre of $x \in \Pp ic^0(\Pp ic^0_{\Cc/S})$ is a $\GG_m$-torsor $\LL_x^*$ over $\Pp ic^0_{\Cc/S}$, isomorphic to $\Gg(\LL_x)$ through the trivialising section $s$.  
Hence, 
$$\Pp_\Cc \setminus 0 \longrightarrow \Pp ic^0(\Pp ic^0_{\Cc/S})$$
is a relative group scheme.

\subsection{Gerbes} \label{sub gerbe}

First, we recall the definition of a gerbe over a stack and the rigidification construction in \cite{ACV}. In this article we are only concerned with $\mu_n$ and $\GG_m$-gerbes. 
\begin{definition}[$G$-gerbe]
    Let $\Mm$ be an algebraic stack and $G$ be a smooth commutative algebraic group over $S$. An étale $G$-gerbe over $\Mm$ is an étale $BG$-torsor over $\Mm$. Equivalently, it is a morphism of stack $\Xx \longrightarrow \Mm$ over $S$ such that
    \begin{itemize}[-]
        \item  For all $U \longrightarrow \Mm$ étale, there exists an étale covering $U' \longrightarrow U$ such that $\Xx_{U'} \neq \emptyset$. 
        \item For all $U \longrightarrow \Mm$ étale, for all objects $s$, $s'$ in $\Xx_U$, there exists an étale covering $V \longrightarrow U$ such that $s_{|V} \simeq s'_{|V}$. 
        \item There is an isomorphism $G_\Mm \xlongrightarrow{\sim} I(\Xx/\Mm)$ where $I(\Xx/\Mm)$ is the inertia stack of $\Xx$ over $\Mm$. 
        \end{itemize}
\end{definition}

\textbf{Rigidifications.} Let $H$ be a smooth commutative algebraic group and $\Mm$ an algebraic stack over $S$. Suppose that for any $T \longrightarrow S$ and $x \in \Mm(T)$, there is an embedding $H_T \hookrightarrow ZAut_T(x)$ compatible with base changes. Then, there exists a rigidified stack $\Mm^{\rig}$ and a $H$-gerbe $\Mm \longrightarrow \Mm^{\rig}$ such that $H_T \subset Aut_T(x)$ is in the kernel of the morphism, and $\Mm^{\rig}$ has the same coarse moduli space as $\Mm$ \cite{ACV}. Moreover, if there is a group $\Gamma$ acting on the stack $\Mm$, the rigidification commutes with taking the quotient \cite[Theorem 5.1]{Romagny}. 

\subsubsection{Torsor of trivialisations} \label{sub torsor of trivial}
 Let $S = \Spec \ k$, such that $n$ is not divisible by the residue field characteristic and $\mu_{n,k}$ is constant. A gerbe is étale locally trivial, and the set of trivialisations has a natural torsor structure.
 
\begin{definition}
    For a $G$-gerbe $\alpha$ over $\Mm$ its torsor of trivialisations $Triv(\alpha \ | \ -)$ is the following sheaf over $\Mm$. For $T \longrightarrow  \Mm$, 
$$Triv(\alpha \ | \ T) = Isom(T \times_{\Mm} \alpha, T \times_S BG). $$
\end{definition}

The work of Giraud gives (equivalence classes of) gerbes an étale $2$-cocycle description \cite{Giraud}. For us, $G$ is a smooth abelian group so that étale and \textit{fppf} topologies coincide \cite[Exposé 6, Théorème 11.7]{Giraud1968DixES}. It implies that $Triv(\alpha \ | \ T)$ is a quasi-torsor under the group $H^1_{\textnormal{ét}}(T, G)$. This construction provides a map 
$$ H^2(T, G) \xlongrightarrow{Triv(- \ | \ T)} H^1(S, H^1(T, G))$$
which is a group morphism. In particular, for $G = \GG_m$ (resp. $G = \mu_n$), $Triv(\alpha \ | \ T)$ is a torsor under $H^1(T, \GG_m) \simeq \Pp ic(T)$ (resp. $H^1(T, \mu_n)$), see \cite{hitchingerbe}. 

The Kummer sequence
$$ 1 \longrightarrow \mu_n \longrightarrow \GG_m \xlongrightarrow{\lambda \mapsto \lambda^n} \GG_m \longrightarrow 1 $$ 
provides a long exact sequence for any $S$-scheme $T$. There are maps : 
\begin{align}
     H^1_{\textnormal{ét}}(T, \mu_n) &\longrightarrow H^1_{\textnormal{ét}}(T, \GG_m) \label{ntorslinebundles} \\
     H^2_{\textnormal{ét}}(T, \mu_n) &\longrightarrow H^2_{\textnormal{ét}}(T, \GG_m). \label{ntorsgerbes}
\end{align} 


\begin{definition}[$Triv^0$] \label{triv0}
    Let $\alpha$ be a $\GG_m$-gerbe of $n$-torsion over $\Mm$, i.e. via (\ref{ntorsgerbes}) such that $\alpha = \alpha_n \times_{B\mu_n} B\GG_m$ where $\alpha_n$ is a $\mu_n$-gerbe. For an irreducible scheme $T \longrightarrow \Mm$, we define :
    $$ Triv^0(\alpha \ \mid \ T) = ( Triv(\alpha_n \ \mid \ T) \times \Pp ic^0(T))/H^1_{\textnormal{ét}}(T, \mu_n)  $$
    where $H^1_{\textnormal{ét}}(T, \mu_n)$ acts diagonally. 
\end{definition}

    The scheme $Triv^0(\alpha \ \mid \ T)$ is a torsor under $\Pp ic^0 (T)$ so it induces an element of $H^1(S, \Pp ic^0(T))$. Equivalently, assuming $T$ complete and normal, the short exact sequence of abelian group schemes \cite[Proposition 4.11]{Milneetale}
$$ 1 \longrightarrow k^*/k^{*n} \longrightarrow H^1_{\textnormal{ét}}(T, \mu_n) \longrightarrow \Pp ic^0(T)[n] \longrightarrow 1$$
induces a long exact sequence, hence a map
  \begin{align} \label{map triv0}
      H^1(S, H^1_{\textnormal{ét}}(T, \mu_n)) \longrightarrow H^1(S, \Pp ic^0(T)), 
  \end{align}  
where we post-composed with the map $H^1(S, \Pp ic^0(T)[n]) \longrightarrow H^1(S, \Pp ic^0(T))$ induced by the isogeny $[n]$ on $\Pp ic^0(T)$. 

The class of $Triv^0(\alpha \ \mid \ T)$ is the image of the class of $Triv(\alpha_n \ \mid \ T)$ in $H^1(F, H^1_{\textnormal{ét}}(T, \mu_n))$ by (\ref{map triv0}). We constructed a group morphism $Triv^0( - \ | \ T)$ in the following commuting square :  
$$\begin{tikzcd}
    H^2(T, \mu_n) \arrow[r, "(\ref{ntorsgerbes})"] \arrow[d, swap, "Triv( - \ | \ T)"]& H^2(T, \GG_m)[n] \arrow[d, "Triv^0( - \ | \ T)"] \\
    H^1(S, H^1_{\textnormal{ét}}(T, \mu_n)) \arrow[r, "(\ref{map triv0})"] & H^1(S, \Pp ic^0(T))  .
\end{tikzcd}$$

\begin{remark}
   The torsor $Triv^0$ is called $Triv^{U(1)}$ in \cite{HT} and $Split'$ in \cite{GWZ}. 
\end{remark}

\subsubsection{Gerbes functions} \label{gerbefunction}

Let $F$ be a $p$-adic field and $S = \Spec F$. We denote $\Oo_F$ its ring of integers. A gerbe $\alpha$ on a stack $\Mm$ gives rise to a $\CC$-valued function \cite[§5]{GWZ} on $\Mm(F)$ 
$$f_{\alpha} \ : \  \Mm(F) \longrightarrow \CC $$
$$ x \longmapsto exp(2\pi i \cdot \operatorname{inv}(x^* \alpha)).$$

For a $F$-point $x$ in $\Mm$, the pullback of the gerbe $x^* \alpha$ represents an element of the Brauer group $H^2(F, \GG_m)$. Then $$inv \ : \ H^2(F, \GG_m) \xlongrightarrow{\sim} \QQ/\ZZ$$ is the Hasse invariant. The construction is explained, for example, in \cite[§31]{Lorenz}. Below we explain the concrete situation in which we use the gerbe functions. 

\begin{Construction}[Gerbe function on the coarse moduli space.] \label{gerbe function construction}
    Let $\Mm \xlongrightarrow{\pi} M$ be a good moduli space map such that there is an open substack $U \rightarrow \Mm$ such that $\pi_{|U}$ is an isomorphism. Suppose that $x \in M(\Oo_F)^\sharp = M(\Oo_F) \cap U(F) $. Let $\alpha$ be a $\GG_m$-gerbe on $\Mm$. We define $$f_\alpha(x) = f_\alpha(\tilde{x})$$
    where $\tilde{x}$ is the lift of $x$ to $\Mm(F)$ via $\pi$. 
\end{Construction}

Let $T \rightarrow S$ be a torsor under an abelian variety $A \rightarrow S$. Then, the gerbe function can be interpreted in terms of the Tate duality pairing. As before, suppose $\alpha$ is a gerbe on $\Mm$ and $T \longrightarrow \Mm$. The torsor of trivialisation $Triv^0(\alpha \ \mid \ T)$ represents a unique element of $H^1(F, \hat{A})$ where $\hat{A} = \Pp ic^0(A)$ is the dual abelian variety of $A$. Using the Tate pairing \cite[Corollary I.3.4]{milnearithmdual}
$$ \langle -,- \rangle_A \ : \  A(F) \times H^1(F, \hat{A}) \longrightarrow \QQ/\ZZ$$
the Hasse invariant of $y^* \alpha$ at $y \in T(F)$ can be expressed as 
\begin{align*}
   \operatorname{inv}(y^* \alpha) = c \cdot \langle h_x(y), Triv^0(\alpha \ \mid \ A)\rangle_A 
\end{align*} 
where we picked $x \in T(F)$, inducing a trivialisation $T \xlongrightarrow{h_x} A $. Here, $c$ denote a $\mathit{n}$th root of unity which depends on this choice of trivialisation \cite[Lemma 6.7]{GWZ}.

\subsection{Moduli stacks of \texorpdfstring{$\GLn$}{GL(n)} and \texorpdfstring{$\SLn$}{SL(n)}-Higgs bundles} \label{slngerbe}

Recall that $\Cc$ is a relative curve over the base scheme $S$. Let $D$ be a relative effective divisor of $\Cc$, flat over $S$ such that $\deg D > 2g-2$ or $D = K_\Cc$.

\begin{notation}
 For $T \longrightarrow S$, we use $\Cc_T = T \times_S \Cc$. The projection maps are denoted $p \ : \ \Cc_T  \longrightarrow T$, $q \ : \ \Cc_T  \longrightarrow \Cc$.
\end{notation}

\begin{notation}
    We use $d$ and $e$ to denote the degree of bundles. In general, we use $d$ for the $\SLn$ side and $e$ for the $\PGLn$ side. Over non-algebraically closed fields, the moduli space of (twisted) $\SLn$-Higgs bundles depends on the choice of a line bundle of degree $d$, which we denote $L$. 
\end{notation}

\subsubsection{The moduli stack of $\GLn$-Higgs bundles} Let $d \in \ZZ$. The underlying category fibred in groupoids of the stack $\Mm_{\GLn}^d$ of semi-stable degree $d$ Higgs bundles over $\Cc$ can be described as follows. Details on the construction of the moduli stack are given in \cite[Remark 7.2]{GWZ} and \cite[Section 7]{Casalaina}. Let $T \longrightarrow S$. Then, 
\begin{align*}
    \Ob(\Mm_{\GLn}^d(T)) = \Bigg\{ (E, \phi) \ &| \ E \in \Ob(\Vv ec_n(\Cc)(T)), \ \phi \in \Hom_{\Oo_{C_T}}(E, E \otimes q^*D), \\
    & \ \text{semi-stable}, \ \det(E) \in \Pp ic_{\Cc/S}^d(T) \Bigg\} 
\end{align*} 
where $\Vv ec_n(\Cc)$ denote the stack of rank $n$ vector bundles on $\Cc$. Morphisms in the groupoids are given by isomorphisms of vector bundles commuting with the Higgs field $\phi$. For a map $T' \longrightarrow T$, there is a pullback functor $\Mm_{\GLn}^d(T) \longrightarrow \Mm_{\GLn}^d(T')$. 

\subsubsection{Moduli stacks of $\SLn$-Higgs bundles}

 Taking the determinant of the bundle and trace of the Higgs field gives an albanese map : 
\[ \Mm_{\GLn}^d \xlongrightarrow{alb} \Pp ic_{C/S}^d \times H^0(\Cc, D). \]

Its fibres are stacks of twisted $\SLn$-Higgs bundles. Precisely, for $L \in \Pp ic_{C/S}^d(S)$ we denote ${\Mm_{\SLn}^L = alb^{-1}(L, 0)}$. Then, for $T \longrightarrow S$,
\begin{align*}
    \Ob(\Mm_{\SLn}^L(T)) = \Bigg\{ (E, \phi) \ &| \ E \in \Vv ec_n(C)(T), \  \phi \in H^0(\Cc, \mathfrak{sl}_n(E)\otimes D), \\
    & \ \text{semi-stable}, \ \det(E) \simeq q^*L \in \Pp ic_{C \times_S T/S}^d  (S) \Bigg\} .
\end{align*} 

\textbf{A $\mu_n$-gerbe.} Each object $(E, \phi)$ over $T$ has an automorphism group $\Aut_{\SLn} (E, \phi)$. The scalar action gives an embedding $\mu_n(T) \hookrightarrow Z(\Aut_{\SLn} (E, \phi))$. The rigidification \cite{ACV} with respect to $\mu_n$ is an algebraic stack $\Mm_{\SLn}^{L, \rig}$ endowed with a $\mu_n$-gerbe : 
$$ \alpha_n \ : \  \Mm_{\SLn}^L \longrightarrow \Mm_{\SLn}^{L, \rig}. $$
When $n$ and $d$ are coprime, the rigidified stack is also the coarse moduli space, which is a smooth quasi-projective variety. 

\textbf{A $\GG_m$-gerbe.} We can also consider the following substack $\Mm_{\widetilde{\SLn}}^L$ of $\Mm_{\GLn}^d$ as a full subcategory given by :  
\begin{align*}
\Ob(\Mm_{\widetilde{\SLn}}^L(T)) = \Big\{ (E, \phi) \in \Ob(\Mm_{\GLn}^d(T)) \ & \big| \  \forall t \in |T|, \ \det(t^*E) \simeq L \in \Pp ic_{C/S}(S), \ \tr \phi = 0 \Big\} 
\end{align*} 
These Higgs bundles are objects of $\Mm_{\SLn}^L(T)$ with a twist by a line bundle over $T$.  Equivalently, 
$$ \Mm_{\widetilde{\SLn}}^L = \Mm_{\SLn}^L \times_{B\mu_n} B\GG_m.$$
For $(E, \phi) \in \Mm_{\widetilde{\SLn}}^L(T)$, let $\Aut'(E, \phi)$ denote its group of automorphisms in $\Mm_{\widetilde{\SLn}}^L(T)$. Its center contains scaling actions, hence a copy of $\GG_{m, T}$. For $T = \Spec k$, 
$$Aut'(E, \phi) = \GG_{m, k} \times^{\mu_{n, k}} \Aut_{\SLn} (E, \phi).$$
 
 Let $\Mm_{\widetilde{\SLn}}^{L, \rig}$ denote the $\GG_m$-rigidification. In $\Mm_{\widetilde{\SLn}}^{L, \rig}$, the twist by a line bundle on $T$ is the identity isomorphism. Hence, 
 $$ \Mm_{\widetilde{\SLn}}^{L, \rig} \simeq \Mm_{\SLn}^{L, \rig}. $$

 Let $\alpha$ denote the $\GG_m$-gerbe \begin{align} \label{slngerbe eq}
     \Mm_{\widetilde{\SLn}}^L \xlongrightarrow{\alpha} \Mm_{\SLn}^{L, \rig}.
 \end{align}
 It is induced by the $\mu_n$-gerbe $\alpha_n$ via (\ref{ntorsgerbes}). Let $M_{\SLn}^{L}$ denote the good moduli space of $\Mm_{\SLn}^{L, \rig}$. We constructed the following commuting diagram : 

$$\begin{tikzcd}
   \Mm_{\widetilde{\SLn}}^L  \arrow[r, "\alpha"] &    \Mm_{\SLn}^{L, \rig} \arrow[r] &  M_{\SLn}^{L} \\
    \Mm_{\SLn}^{L}   \arrow[ru, swap, "\alpha_n"]   \arrow[u] & &  
\end{tikzcd}$$

\subsection{Stacks of \texorpdfstring{$\PGLn$}{PGL(n)}-Higgs bundles} \label{pglngerbe}

\subsubsection{The \texorpdfstring{$\Gamma$}{Gamma}-action.} \label{gamma action}

We follow Romagny's definition of a group action on a stack \cite{Romagny}. 

\begin{definition}[Group action on a stack \cite{Romagny}] \label{Romagnydef}
    Let $G$ be a group scheme over $S$, with multiplication $m$ and neutral element map $e$. A group action on a stack $\Mm$ is a morphism of stacks 
    $$ \mu \ : \ G \times \Mm \longrightarrow \Mm $$
    such that
    
    $$ \begin{tikzcd}[row sep=large, column sep=large]
       G \times G \times \Mm \arrow[r, "m \times \id_\Mm"] \arrow[d, "\id_G \times \mu"] & G \times \Mm \arrow[d, "\mu"] \\
        G \times \Mm \arrow[ru, Rightarrow, "\alpha"] \arrow[r, "\mu"] & \Mm  
    \end{tikzcd}
    \quad 
    \quad
    \begin{tikzcd}[row sep= small, column sep=4em]
        G \times \Mm \arrow[rr, "\mu"] 
        \arrow[rd, Rightarrow, "\beta"] & & \Mm \\
        & \phantom{\Mm} & \\
         \Mm \arrow[uu, "e \times \id_\Mm "]  \arrow[from=3-1, to=1-3, swap, "\id_\Mm"]& &
    \end{tikzcd}$$

    is $1$-commutative, meaning that $\alpha$ and $\beta$ are the identity $2$-morphisms. 
    
\end{definition}

Fix $N\in \Pp ic^1_{\Cc/S}(S)$. Recall that $\Gamma = \Pp ic^0(\Cc)[n] \simeq (\ZZ/n\ZZ)^{2g}$. For all $\gamma \in \Gamma$, we can use the normalisations of the Poincaré bundle to equip $g_N^* \LL_\gamma$ (see Notations \ref{notation LL et L} and \ref{gN}) with a non-zero vector $v_\gamma$ in the fibre of $N$. We can then set $\xi^N_{\id}$ to be a trivialising section of $\Oo_{\Pp ic^1_{\Cc/S}}$ containing $v_{\id}$ and for $\gamma, \gamma' \in \Gamma$, 
$$ \xi^N_{\gamma, \gamma'} \ : \ g_N^*\LL_\gamma \otimes g_N^*\LL_{\gamma'} \xlongrightarrow{\sim} g_N^*\LL_{\gamma \gamma'} $$ sending $(v_\gamma \otimes v_{\gamma'})$ to $v_{\gamma\gamma'}$. 

Pulling back $\xi^N_{\gamma, \gamma'}$, $\xi^N_{\id}$ along the Abel-Jacobi map we obtain a trivialising section $\psi^N_{\id}$ of $ \Oo_\Cc$
and
$$ \psi^N_{\gamma, \gamma'} \ : \ L_\gamma \otimes L_{\gamma'} \xlongrightarrow{\sim} L_{\gamma \gamma'}. $$

\begin{lemma}[$\Gamma$-action on $\Mm_{\widetilde{\SLn}}^{L_e}$]
    Let $e \in \ZZ$, $L_e \in \Pp ic^e_{\Cc/S}(S)$ and recall the stack $\Mm_{\widetilde{\SLn}}^{L_e}$ from \autoref{slngerbe}. Let $T, T'$ be $S$-schemes. 
\begin{itemize}[-]
    \item Let $F_\gamma^T \ : \ \Mm_{\widetilde{\SLn}}^{L_e}(T) \longrightarrow \Mm_{\widetilde{\SLn}}^{L_e}(T)$ be the functor defined by :
    
    for $(E, \phi)$, $(E', \phi')$ objects in $\Mm_{\widetilde{\SLn}}^{L_e}(T)$,
    $$F_\gamma^T((E, \phi)) = (q^*L_\gamma \otimes E, \phi \otimes \id),$$
    
     and for $f \in Hom_{\Mm_{\widetilde{\SLn}}^{L_e}(T)}((E, \phi),(E', \phi'))$, $F_\gamma (f) = \id_{q^*L_\gamma} \otimes f$ .

    \item Let  $F_\gamma^T \longrightarrow F_\gamma^{T'}$ be the natural transformation induced by pullback along $T' \longrightarrow T$. 

    \item For any $(E, \phi) \in \Mm_{\widetilde{\SLn}}^{L_e}(T)$, let
$$  f_{\id}^N  \ : \  E \otimes q^*\Oo_\Cc \xlongrightarrow{\sim} E$$
$$ f_{\gamma, \gamma'}^N \ : \ q^*L_\gamma \otimes q^*L_{\gamma'} \otimes E \xlongrightarrow{\sim} q^*L_{\gamma \gamma'} \otimes E $$
 be induced by $\psi^N_{\id}$ and $\psi^N_{\gamma, \gamma'}$.
\end{itemize}
    This data defines an action as in \cite{Romagny} :
    $$ a_N \ : \ \Gamma \times \Mm_{\widetilde{\SLn}}^{L_e}  \longrightarrow  \Mm_{\widetilde{\SLn}}^{L_e}.$$ 
\end{lemma}

\begin{proof}
Let $T \longrightarrow S$ and $(E, \phi) \in \Mm_{\widetilde{\SLn}}^{L_e}(T)$.   The $1$-commutativity in Definition \ref{Romagnydef} requires identifications  $$Hom_{\Mm_{\widetilde{\SLn}}^{L_e}}(E, E \otimes q^*\Oo_\Cc) \simeq Hom_{\Mm_{\widetilde{\SLn}}^{L_e}}(E, E) $$ and $$Hom_{\Mm_{\widetilde{\SLn}}^{L_e}}(E, q^*L_{\gamma \gamma'} \otimes E) \simeq Hom_{\Mm_{\widetilde{\SLn}}^{L_e}}(E, q^*L_\gamma \otimes q^*L_{\gamma'} \otimes E) .$$ 
These identifications can be made using isomorphisms : 
$$  f_{\id}^N  \ : \  E \otimes q^*\Oo_\Cc \xlongrightarrow{\sim} E$$
$$ f_{\gamma, \gamma'}^N \ : \ q^*L_\gamma \otimes q^*L_{\gamma'} \otimes E \xlongrightarrow{\sim} q^*L_{\gamma \gamma'} \otimes E. $$
Since they are induced by $\psi^N_{\id}$ and $\psi^N_{\gamma, \gamma'}$, they are compatible with isomorphisms in $\Mm_{\widetilde{\SLn}}^{L_e}(T)$ and pullbacks along $T' \longrightarrow T$. 
    \end{proof}

\subsubsection{The $\PGLn$-stack and its gerbe}
Let $e \in \ZZ$ and $L_e \in \Pp ic^e(\Cc)$. We consider the following quotient stack whose existence and algebraicity follows from \cite[Proposition 2.6, Theorem 4.1]{Romagny} :
$$ \Mm_{\PGLn}^e = \Mm_{\SLn}^{L_e, \rig}/\Gamma =  \Mm_{\GLn}^{e, \rig}/\Pp ic^0_{\Cc/S}.$$
It does not depend on $L_e$. The tensorisation action by line bundles of $\Gamma$ on $\Mm_{\SLn}^{L_e, \rig}$ is well defined without a choice of $N \in \Pp ic^1_{\Cc/S}$, since scalar automorphisms are all identified. The good moduli space is $M^e_{\PGLn} = M_{\SLn}^{L_e}/\Gamma$. 

We now turn to the construction of a ($n$-torsion) $\GG_m$-gerbe on the stack $\Mm_{\PGLn}^e$. A gerbe on $\Mm_{\PGLn}^e$ is a $\Gamma$-equivariant gerbe on $\Mm_{\SLn}^{L_e, \rig}$. The action $a_N$ equip the gerbe $\alpha \ : \  \Mm_{\widetilde{\SLn}}^{L_e}  \longrightarrow \Mm_{\SLn}^{L_e, \rig}$ with a $\Gamma$-equivariant structure. We write $\alpha_N$ for this $\Gamma$-equivariant gerbe. It descends to a gerbe on the quotient $\Mm_{\PGLn}^e$ which we still denote
\begin{align} \label{pgln gerbe eq}
    \alpha_N \colon \Mm_{\widetilde{\SLn}}^{L_e}  \longrightarrow \Mm_{\PGLn}^e. 
\end{align} 


\subsubsection{The torsor of $\Gamma$-equivariant trivialisations}

Let $S = \Spec k$. Let $N \in \Pp ic^1_{\Cc/S}$. We construct a relative group scheme $\Gamma_N \longrightarrow \Cc$. In \autoref{subsection notations picard} we recalled the relative group scheme $\Pp_\Cc \setminus 0 \longrightarrow \Pp ic^0(\Pp ic^0_{\Cc/S})$, whose fibres are total spaces of line bundles over $\Pp ic^0_{\Cc/S}$. As $\Gamma \hookrightarrow \Pp ic^0_{\Cc/S} $, the fibrewise restriction is a relative subgroup 
$$\Pp_\Cc \setminus 0_{\mid \Gamma} \longrightarrow  \Pp ic^0(\Pp ic^0_{\Cc/S}).$$

Pulling back along the self duality isomorphism, the underlying scheme of $\Pp \setminus 0_{\mid \Gamma}  \longrightarrow \Pp ic^0(\Pp ic^0_{\Cc/S})$ can be identified with $$\coprod_{\gamma  \in \Gamma} \LL_\gamma^* \longrightarrow \Pp ic^0_{\Cc/S}, $$ 
giving $\underset{\gamma  \in \Gamma}{\coprod} \LL_\gamma^*$ a relative group scheme structure. 

\begin{Construction}[The group $\Gamma_N$]
    Using the composition of the Abel-Jacobi map and $g_N$ (Notation \ref{gN}), we can pull back the relative group scheme $\underset{\gamma  \in \Gamma}{\coprod} \LL_\gamma^* \longrightarrow \Pp ic^0_{\Cc/S}$ over $\Cc$. We call this relative group $\Gamma_N$, whose underlying scheme structure is isomorphic to $$ \Gamma_N = \coprod_{\gamma} L_\gamma^* \longrightarrow \Cc.$$
\end{Construction}

\begin{remark}
    The group scheme structure of $\Gamma_N$ does not come from the theta group of $\LL_\gamma$. Instead, for $c \in \Cc$, it comes from the theta group of $\Oo_{\Pp ic^0_{\Cc/S}} ([\Oo_\Cc(c)] - [0])$. It is the subgroup given by restriction along $\Gamma \subset \Pp ic^0_{\Cc/S}$. 
\end{remark}
\begin{remark}
  The group $\Gamma_N$ contains more information than its underlying set $\underset{\gamma  \in \Gamma}{\coprod} L_\gamma \setminus \{0\}$. The chosen trivialising section of the zero fibre of the Poincaré sheaf induces a canonical splitting  $$[N]^*{g_N^* \coprod_{\gamma} L_\gamma^*} \cong \Gamma \times_S \GG_m,$$ 
  sending the rational points $v_\gamma$ (introduced in \ref{gamma action}) to $(\gamma, 1)$. 

  If $c\in \Cc$ and $N = \Oo_{\Cc}(c)$, then $c^* \Gamma_N \cong  \Gamma \times_S \GG_m$. In this context, the group $\Gamma_N$ is called $\Tilde{\Gamma}$ in \cite{HT}. 
\end{remark}

There is a projection $ \Gamma_N \longrightarrow \Gamma$. In fact, $\Gamma_N$ is a non-trivial extension of relative group schemes over $\Cc$,
    $$ 1 \longrightarrow \GG_m \xlongrightarrow{\psi_1^N} \Gamma_N \longrightarrow
 \Gamma \longrightarrow 1.$$

The relative group scheme $\Gamma_N \rightarrow \Cc$ allows us to characterise the trivialisations of the gerbe $\alpha_N$ over $\Mm^e_{\PGLn}$. 

\begin{lemma}[Trivialisations of $\alpha_N$] \label{equiv lemma}
   A trivialisation over a $\Gamma$-scheme $T \longrightarrow \Mm_{\SLn}^{L_e, \rig}$ of $\alpha_N$ as a $\Gamma$-equivariant gerbe is a $\Gamma_N$-equivariant object in $\Mm_{\widetilde{\SLn}}^{L_e}(T)$ with $\GG_m \hookrightarrow \Gamma_N$ acting with weight $1$.
\end{lemma}

\begin{proof}
 A $\GG_m$-gerbe is a $B\GG_m$-torsor : it is étale-locally trivial, and a local trivialisation is given by a section. A $\Gamma$-equivariant gerbe is a $\Gamma$-equivariant $B\GG_m$-torsor, hence a local trivialisation is given by a $\Gamma$-equivariant section over a $\Gamma$-scheme. 

 Let $T \longrightarrow S$ be a $\Gamma$-scheme, where $t_\gamma  \ : \ T \longrightarrow T$ denote the action of $\gamma \in \Gamma$.  Let $T \longrightarrow \Mm_{\SLn}^{L, \rig}$ be $\Gamma$-equivariant. It corresponds to a projective Higgs bundle $(\PP\EE, \Phi)$ on $T \times_S \Cc$ of determinant $[q^*L] \in \Pp ic_{\Cc/ S}(T) = \Pp ic(T \times_S \Cc)/\Pp ic(T) $. 
 
A $\Gamma$-equivariant section of $\alpha_N$ is a lift of $\PP\EE$ to a vector bundle $\EE \longrightarrow T \times \Cc$ of determinant $q^*L \in \Pp ic(T \times_S \Cc)$, with a Higgs field, such that for every $\gamma \in \Gamma$, for every $t \in |T|$ there is an isomorphism $$g_{\gamma, t} \ : \ t^*(\EE  \otimes q^*L_\gamma) \xrightarrow{\sim}  t^* t_\gamma^* \EE$$ 
as vector bundles over $\Cc$, compatible with the Higgs field and such that for $\gamma, \gamma' \in \Gamma$ : 

$$\begin{tikzcd}
    t^*(\EE \otimes q^*L_\gamma \otimes q^*L_{\gamma'}) \arrow[r, "\id \otimes g_{\gamma, t}"] \arrow[d, "\id \otimes f^N_{\gamma, \gamma'}"] &  t^* t_\gamma^* (\EE \otimes q^*L_{\gamma'})  \arrow[d, "g_{\gamma', t}"] \\
     t^*(\EE \otimes q^*L_{\gamma\gamma'}) \arrow[r, "g_{\gamma\gamma', t}"] & 
     t^* t_{\gamma\gamma'}^*\EE 
\end{tikzcd}$$

Since there is a projection $\Gamma_N \longrightarrow \Gamma$, $\Gamma_N$ acts on $T \times_S \Cc$ via the $\Gamma$-action on $T$. Hence, there is a commuting diagram 

$$\begin{tikzcd}
    \Gamma_N \times_\Cc \EE \arrow[r, "f"] \arrow[d] & \underset{\gamma \in \Gamma}{\coprod} q^*L_\gamma \otimes \EE  \arrow[r, "g"] & \EE \arrow[d] \\
    \Gamma \times T \times_S \Cc \arrow[rr, "{(\gamma, t, c) \mapsto( \gamma \cdot t, c)}"] & & T \times_S \Cc
\end{tikzcd}$$
where the map $f$ is induced by the identification  $\Gamma_N \simeq  \underset{\gamma \in \Gamma}{\coprod} L_\gamma \setminus \{0\}$ and the natural maps $L_\gamma \times_{\Cc} \EE \longrightarrow q^*L_\gamma \otimes \EE$. The map $g$ is induced by maps $g_{\gamma, t}$ above. For $\gamma, \gamma' \in  \Gamma$, the construction of $\Gamma_N$ and $f_{\id}^N$, $f^N_{\gamma, \gamma'}$ use the same chosen rational point $v_\gamma$ in $\LL_\gamma$. Hence, the product in $\Gamma_N$ is compatible with $g$ so that the diagram above makes $\EE$ into a $\Gamma_N$-equivariant bundle over $T \times_S \Cc$.

 Moreover, $\GG_m \hookrightarrow \Gamma_N$, and using the compatibility between $g_{\id, t}$ and $f^N_{\id}$, for $\lambda \in \GG_m$ the upper map sends $(\lambda, v )$ to $\lambda v$, so $\GG_m$ acts with weight $1$. 
\end{proof}

\subsection{Hitchin fibration and spectral curves} \label{pryms}

The coarse moduli spaces $M_{\SLn}^L$ and $M_{\PGLn}^e$ admit a morphism to the affine scheme $\Aa = \underset{{i=2}}{\overset{n}{\bigoplus}} H^0(\Cc, iD)$ over $S$. It sends a Higgs bundle to the coefficients of its characteristic polynomial. This morphism is proper \cite[Thm 6.11]{Simpsonproperness} and is called the Hitchin fibration. Moreover there is an associated relative spectral curve $\Tilde{\Cc}_\Aa \longrightarrow \Aa$. 

We denote $\Aa^{\textnormal{sm}}(S)$ the locus of smooth and irreducible spectral curves. We fix a point $a \in \Aa^{\textnormal{sm}}(S)$ corresponding to a spectral curve $\tilde{\Cc} \longrightarrow \Cc$.

\subsubsection{The $\SLn$-fibration} Let $L \in \Pp ic^d(\Cc)$. We denote $$h_{\SLn} \ : \ M^L_{\SLn} \longrightarrow \Aa$$ the relative Hitchin fibration over $S$. 

The branched cover $\pi \ : \ \Tilde{\Cc} \longrightarrow \Cc$ over $S$ gives rise to a relative Norm map 
$$\Nm_{\Tilde{\Cc}/\Cc/S} \ : \ \Pp ic_{\Tilde{\Cc}/S} \longrightarrow \Pp ic_{\Cc/S}.$$  Let us recall the construction, following \cite[§3]{HP}. Given a line bundle $\Ll$ over $\Tilde{\Cc}$, $\pi_* \Ll$ is an invertible $\pi_* \Oo_{\Tilde{\Cc}}$-module and $\pi_* \Oo_{\Tilde{\Cc}}$ is a locally free rank $n$ sheaf with trivialising open cover $(U_{i})$. Now, take a $1$-cocycle $(\phi_{ij})_{i,j}$ representing the class of $\pi_*\Ll$ in $H^1( \Cc, (\pi_* \Oo_{\Tilde{\Cc}})^*) $. For each $i,j$, consider the determinant of the multiplication by $\phi_{ij}$ in the free $\Oo_\Cc$-module $\pi_* \Oo_{\Tilde{\Cc}}(U_{ij})$. This gives a $1$-cocycle valued in $\Oo_{\Cc}^*$, and its class in $ \Pp ic_{\Cc/S}(S)$, called the norm of $\Ll$, depends only on the class of $\Ll$ in $ \Pp ic_{\Tilde{\Cc}/S}(S)$. 

The norm verifies the following identity \cite[Cor 3.12]{HP} 
$$ \det(\pi_* N) = \Nm_{\Tilde{\Cc}/\Cc/S}(N) \otimes  \det(\pi_* \Oo_{\Tilde{\Cc}})^{-1}. $$

The pullback $\pi^* \colon \Pp ic^0_{\Cc/S} \longrightarrow \Pp ic^0_{\Tilde{\Cc}/S}$ is injective\footnote{We can give a quick argument in our case. Let $M \in \ker \pi^*$. By projection formula, $\pi_* \Oo_{\tilde{\Cc}} \simeq M \otimes \pi_* \Oo_{\tilde{\Cc}}$. Since $\Oo_\Cc$ is a factor of $\pi_* \Oo_{\tilde{\Cc}}$, it means that $M$ should be a factor of $\pi_* \Oo_{\tilde{\Cc}}$. But all of its factor are of the form $\Oo_{\Cc}(nD)$. As $D$ is effective, $M= \Oo_{\Cc}$.}. The norm map is dual to $\pi^*$, so the kernel of $\Nm_{\Tilde{\Cc}/\Cc/S} $ is connected. It is an abelian $S$-scheme denoted $\Pp rym^0(\Tilde{\Cc}/\Cc/S)$. Since $\Tilde{\Cc}$ is fixed here, we just denote it $\Pp rym^0$. 

We use $\Pp rym^{M}$ to denote the fibre $\Nm_{\Tilde{\Cc}/\Cc/S}^{-1} (M)$. It is a $\Pp rym^0$-torsor since $Nm$ is a morphism of abelian schemes. Since $\Tilde{\Cc}$ is embedded in $Tot\Oo_{\Cc}(D)$, $\pi^*\Oo_{\Cc}(D)$ carries a tautological section $\lambda$. The spectral correspondence implies : 
\begin{align*}
   \Pp rym^{L \otimes \Dd} &\xlongrightarrow{\sim}  h_{\SLn}^{-1}(a)  \\
   \Ll &\longmapsto (\pi_* \Ll, \pi_* \lambda)
\end{align*} 
where $ \Dd = \det(\pi_* \Oo_{\Tilde{\Cc}})$. Thus, the Hitchin fibre $h_{\SLn}^{-1}(a)$ is a $\Pp rym^0$-torsor over $S$.

\begin{remark}
If $n = 2k + 1$, there is a $\mathit{n}$th root of $\Dd$ given by $\Oo_\Cc(-D)^{k}$. Similarly if $\deg D$ is even, and if there exists a square root $D^{1/2}$, then $\Dd$ has a $\mathit{n}$th root $(\Oo_\Cc(-D^{1/2}))^{n-1}$. 
In these cases, there is an isomorphism of $\Prym^0$-torsors
$$ \Pp rym^{L \otimes \Dd} \xrightarrow{\sim}  \Pp rym^{L} $$
which explains why $d'$ does not appear in \cite{HT, GWZ}. 
\end{remark}

\subsubsection{The $\PGLn$-fibration}
Consider the Hitchin fibration 
 $$h_{\PGLn} \ : \ M^e_{\PGLn} \longrightarrow \Aa.$$

 As shown in \cite{HT} and \cite{GWZ}, the dual abelian variety of $\Pp rym^0$ is the quotient $\Pp rym^0/\Gamma$. 
 
The spectral correspondence implies that the $\PGLn$-Hitchin fibre of $a \in \Aa(S)$ is a $\Pp rym^0/\Gamma$-torsor, denoted $\Pp rym^{e'}_{\PGLn}$. For any $L_e \in \Pp ic^{e}(\Cc)$, one has 
$$ \Pp rym^{e'}_{\PGLn}=  \Pp rym^{L_e \otimes \Dd}/\pi^* \Gamma = \Pp ic^{e'}_{\Tilde{\Cc}/S} /\pi^* \Pp ic^0_{\Cc/S}.$$

\begin{notation} \label{notation d'}
    We write :
    \begin{itemize}[$\cdot$]
        \item $\Dd = \det(\pi_* \Oo_{\Tilde{\Cc}}) = \overset{n-1}{\underset{i=0}{\bigotimes}} \Oo_\Cc(-iD) = \Oo_\Cc(D)^{-\frac{n(n-1)}{2}}$ (see \cite{klingSO, BNR89} for proofs).
        \item $d' = d + \deg \Dd = d - \deg D \cdot \frac{n(n-1)}{2}$
        \item $e' =  e + \deg \Dd = e - \deg D \cdot \frac{n(n-1)}{2}$
    \end{itemize}
\end{notation}

\subsection{Regularity and stringy twisted invariants}
\label{section symplectic}

\subsubsection{Smoothness and normality}

Let $S = \Spec \CC$. The cotangent complex of the twisted moduli stack $\Mm_{\SLn}^L$ is well understood. At $(E, \phi)$, it is quasi-isomorphic to the hypercohomology of 
$$ End_0(E, \phi) \xrightarrow{ad(\phi)}  End_0(E, \phi) \otimes D, $$
which satisfies Serre duality. In particular, 
\begin{itemize}[-]
    \item When $D > K_\Cc$, $\Mm_{\SLn}^L$ is smooth.
    \item Whem $D \sim K_\Cc$, moduli stacks of (twisted) $G$-Higgs bundles are symplectic \cite[\textsection 10.3.24]{Davison}. Thus, they admit local models by totally negative quivers with potential, hence have rational singularities \cite{Tanguy}. In particular $\Mm_{\SLn}^L$ is a normal stack. 
\end{itemize}

\subsubsection{Stringy twisted invariants}
We refer to \cite[Section 2]{GWZ} and \cite{HT} for stringy twisted invariants.

 Based on a transgression construction and a choice of embedding $\mu_n \hookrightarrow \QQ_\ell$, a $n$-torsion $\GG_m$-gerbe $\alpha$ on a stack $\Xx$ is associated with an $\ell$-adic local system $L_\alpha$ on the inertia stack $I\Xx$. The stringy twisted invariants are invariants on $I\Xx$ with respect to coefficients twisted by $L_\alpha$.





Let $X$ be a variety acted on by a finite abelian group $\Gamma$. Let $Y$ be the good moduli space of the stack $X/\Gamma$. Let $\alpha$ be a $|\Gamma|$-torsion $\GG_m$-gerbe on $X/\Gamma$. It gives rise to a $\ell$-adic local system $L_\alpha$ on the inertia stack $I(X/\Gamma) = \underset{\gamma \in \Gamma}{\coprod} X^\gamma/\Gamma$, whose restriction to $X^\gamma/\Gamma$ is denoted $L_{\alpha, \gamma}$. 

Moreover, assume for simplicity that the fixed point sets $X^\gamma$ are connected. It is the case for $M^L_{\SLn}$ when $(d,n)=1$. Let $F_\gamma$ be a set of rational numbers indexed by $\Gamma$ (called fermionic shifts). Let $(X^\gamma/\Gamma (k))_{iso}$ denote the isomorphism classes of $k$-points of the stack $X^\gamma/\Gamma$. 

\begin{definition}[stringy, twisted invariants] \label{defi stringy}
  In the situation above with $X/\Gamma \longrightarrow Y$, a gerbe $\alpha$ on $X/\Gamma$ and fermionic shifts $\{F_\gamma\}_{\gamma \in \Gamma}$ :
\begin{enumerate}[(a)]

    \item \textbf{(stringy twisted point counts)}  Let $S = \Spec \FF_q$. 
    $$ \#_{\textnormal{st}, \alpha} (Y) = \sum_{\gamma \in \Gamma} \ \sum_{x \in X^\gamma/\Gamma (\FF_q)_{iso}} q^{F_\gamma} \frac{\Tr (\Fr \ | \ (L_{\alpha, \gamma})_x) }{|Aut(x)|}. $$
    \item \textbf{(stringy twisted $E$-polynomials)} Let $S = \Spec \CC$. 
    $$ E_{\textnormal{st}, \alpha} (Y) = \sum_{\gamma \in \Gamma} \sum_{p, q \in \ZZ}  \dim H_c^{p,q} (X^\gamma, L_{\alpha, \gamma})^\Gamma u^p v^q (uv)^{F_\gamma},  $$
    where we use compactly supported cohomology. 
\end{enumerate}
\end{definition}

In the case of $M^e_{\PGLn} \simeq M^{L_e}_{\SLn}/\Gamma$, we write $M^{L_e, \gamma}_{\SLn}$ for the $\gamma$-fixed points. The fermionic shifts are given by $$F_\gamma = 2 \cdot  codim_{\Aa}(\Aa^\gamma) = codim_{M^{L_e}_{\SLn}}(M^{L_e, \gamma}_{\SLn})$$ 
where $\Aa^{\gamma}$ is the image of $M^{L_e, \gamma}_{\SLn}$ by the Hitchin map. It does not depend on $L_e$ (see \cite[Lemma 5.6]{HPL} for a computation).



\section{SYZ-type mirror symmetries} \label{section SYZ}



Set $S = \Spec k$ where $k$ is a field with $ \charac k \nmid n $, with $\mu_{n, k}$ and $\Gamma = \Pp ic^0_{\Cc/S} [n]$ constant group schemes. 

Let $a \in \Aa^{\textnormal{sm}}(k)$, corresponding to a spectral curve $\Tilde{\Cc} \xlongrightarrow{\pi} \Cc$.

\subsection{First \texorpdfstring{$SYZ$}{SYZ}-type symmetry}

The following proposition generalises \cite[Proposition 3.6]{HT} to any field $k$ satisfying our assumptions. 

\begin{proposition}[First $SYZ$-type symmetry]
\label{firstSYZ}
Let $d \in \ZZ$ and $L \in \Pp ic^d_{\Cc/S}(S)$. Let $\alpha$ denote the gerbe $\Mm^L_{\SLn} \rightarrow \Mm^{L, \rig}_{\SLn}$ (\ref{slngerbe eq}). Set $L' = L \otimes \Dd$ (Notation \ref{notation d'}). For any $e \in \ZZ$, there is an isomorphism of $\Pp rym^0/\Gamma$-torsors over $S$ :
    $$ \Triv^0 (\alpha^{e} \ | \ \Pp rym^{L'}) \simeq \Pp rym^{e}_{\PGLn}. $$
\end{proposition}

\begin{proof}
    We can follow closely the proof by Hausel and Thaddeus. First, the statement for $e=1$ suffices, since the map $Triv^0$ in \ref{sub torsor of trivial} is a group morphism, and we can iterate the group law for $\Pp rym^0/\Gamma$-torsors over $S$ on both sides. 
    
  A trivialisation of the $\mu_n$-gerbe $\Mm_{\SLn}^{L} \longrightarrow \Mm_{\SLn}^{L, \rig}$ over $ \Pp rym^{L'} \longrightarrow \Mm_{\SLn}^{L, \rig}$ is a section $$ \Pp rym^{L'} \longrightarrow \Pp rym^{L'} \times_{\Mm_{\SLn}^{L, \rig}} \Mm_{\SLn}^{L}.$$ 
    Using the spectral correspondence, it is a universal line bundle  $\Ll$ over $\Pp rym^{L'} \times \tilde{\Cc}$ such that $\det \pi_* \Ll \simeq pr_2^*L$. 
    
    Let $\Tt = \Triv^0 (\alpha \ | \ \Pp rym^{L'})$. It is the set of universal line bundles over $\Pp rym^{L'} \times \tilde{\Cc}$ such that any pullback over $\Pp rym^{L'}$ is of degree $0$. The universal property implies that two such bundles differ by an element of $\Pp ic^0 (\Pp rym^{L'} ) \simeq \Pp rym^0/\Gamma$. 
    
    This  $\Pp rym^0/\Gamma$-torsor $\Tt$ is a quotient by $\Pp ic^0_{\Cc/S}$ of the $\Pp ic^0_{\Tilde{\Cc}/S}$-torsor 
    $$\Tt' = \big\{ \Ll \longrightarrow \Pp ic^{d'}_{\Tilde{\Cc}/S} \times \Tilde{\Cc} \text{ universal } \mid \ y^* \Ll \in \Pp ic^0(\Pp ic^{d'}_{\Tilde{\Cc}/S}) \ \forall \ y \in |\tilde{\Cc}| \big\}. $$ 
    In fact, 
    \begin{itemize}[-]
        \item Since $\Pp ic^0_{\Cc/S} \cong \Pp ic^0_{\Pp ic^{d'}_{\Cc/S}/S}$, the action of $\Pp ic^0_{\Cc/S}$ is by pullback of line bundles over $\Pp ic^{d'}_{\Cc/S}$ along
   $$ \begin{tikzcd}
       \Pp ic^{d'}_{\Tilde{\Cc}/S} \times \Tilde{\Cc}  \arrow[r, "pr_1"] & \Pp ic^{d'}_{\Tilde{\Cc}/S} \arrow[r, "\Nm" ]& \Pp ic^{d'}_{\Cc/S}.
    \end{tikzcd}$$
    \item There is a natural map $\Tt' \longrightarrow \Tt$ given by restriction, since $\Prym^L \hookrightarrow \Pp ic^{d'}_{\tilde{\Cc}/S}$. It is compatible with the $\Pp ic^0_{\tilde{\Cc}/S}$-action, and since $\Tt$ is a  $\Pp ic^0_{\tilde{\Cc}/S}/\Pp ic^0_{\Cc/S}$-torsor, it is $\Pp ic^0_{\Cc/S}$-invariant. Hence, it induces a map $$\Tt'/\Pp ic^0_{\Cc/S} \longrightarrow \Tt$$ of $\Pp rym^0/\Gamma$-torsors, which must be an isomorphism. 
    \end{itemize}

Moreover, $\Pp rym^{1}_{\PGLn}$ is also a quotient by $\Pp ic^0_{\Cc/S}$ of a $\Pp ic^0_{\Tilde{\Cc}/S}$-torsor, namely $\Pp ic^1_{\Tilde{\Cc}/S}$. Hence, it suffices to show an isomorphism of $\Pp ic^0_{\Tilde{\Cc}/S}$-torsors $ \Tt' \xlongrightarrow{\sim} \Pp ic^1_{\Tilde{\Cc}/S} $.

Let $T$ be a $S$-scheme. For $L_1 \ : \ T \longrightarrow  \Pp ic^1_{\Tilde{\Cc}/S}$, the pullback of the Poincaré line bundle $\Pp_{\Tilde{\Cc}}$ along 

$$\begin{tikzcd}
   \Tilde{\Cc} \times \Pp ic^{d'}_{\Tilde{\Cc}/S} 
 \arrow[r, "h_{L_1}"] & \Pp ic^0_{\Tilde{\Cc}/S} \times \Pp ic^0(\Pp ic^0_{\Tilde{\Cc}/S} ) \\ 
  (y, L_{d'})  \arrow[r] & \big( \Oo_{\Tilde{\Cc}}(y) \otimes L_1^{-1}, \Oo_{\Pp ic^0_{\Tilde{\Cc}/S} }( [L_{d'} \otimes L_1^{-d'}] - [0]) \big)
\end{tikzcd}$$

is a universal line bundle over $(\Tilde{\Cc} \times \Pp ic^{d'}_{\Tilde{\Cc}/S}) \times_S T $, whose pullback over $\Pp ic^{d'}_{\Tilde{\Cc}/S}$ is of degree $0$. It shows that $\Pp ic^1_{\Tilde{\Cc}/S}$ maps to $\Tt'$ via
$$L_1 \longmapsto  h_{L_1}^* \Pp_{\Tilde{\Cc}}.$$  
As both are $\Pp ic^0_{\Tilde{\Cc}/S}$-torsor, they are isomorphic. 
\end{proof}

\begin{remark}[]
For the case $L = \Dd^{-1}$, the $\GG_m$-gerbe $\Mm_{\widetilde{SLn}}^{\Dd^{-1}} \longrightarrow \Mm_{\SLn}^{\Dd^{-1}, \rig} $ restricted to a Hitchin fibre can be written as an extension of commutative group stacks : 
    $$ 0 \longrightarrow B\GG_m \longrightarrow \PP rym^0 \longrightarrow \Pp rym^0  \longrightarrow 0 $$
where $\PP rym^{0}$ is the substack of the Picard stack $\PP ic^0_{\Tilde{\Cc}/S}$ of line bundles $L$ such that $\Nm(L) = \Oo_\Cc$. 
Applying the dualising functor $\Hh om(-, B\GG_m)$, we get an exact sequence (see \cite[§7]{Brochard}) : 
     $$ 0 \longrightarrow Prym^0/\Gamma \longrightarrow Pic(\Tilde{\Cc})/Pic(\Cc)  \xrightarrow{\deg} \ZZ \longrightarrow 0 $$
which corresponds to a $Prym^0/\Gamma$-torsor, splitting when $\Pp ic^1(\Tilde{\Cc})/\Pp ic(\Cc) = \Pp rym^{1}_{\PGLn} $ has a section.

\end{remark}

\subsection{Second \texorpdfstring{$SYZ$}{SYZ}-type symmetry} 

Fix $N \in \Pp ic^1_{\Cc/S}(S)$. The following extends \cite[Proposition 3.6]{HT}, where $N$ is taken to be $\Oo_{\Cc}(c)$ for a basepoint $c \in \Cc$, to more general fields. 

Let $e \in \ZZ$ and $L_e \in \Pp ic^{e}_{\Cc/S}(S)$. Recall that we constructed a $\Gamma$-equivariant gerbe $\alpha_N \colon \Mm_{\widetilde{\SLn}}^{L_e} \longrightarrow \Mm_{\SLn}^{L_e, \rig}$ in \autoref{pglngerbe}. 

\begin{proposition}[Second $SYZ$-type symmetry] \label{secondSYZ} For any $d \in \ZZ$, there is an isomorphism of $\Pp rym^0$-torsors: 
    $$\Triv^0(\alpha_N^d \ | \ \Pp rym_{\PGLn}^{e'}) \simeq \Pp rym^{N^d}$$
where $e' =  e - \deg D \cdot \frac{n(n-1)}{2}$. 
\end{proposition}

\begin{proof}
We follow the arguments of Hausel and Thaddeus again. As for Proposition \ref{firstSYZ}, we first reduce to the case where $d=1$. 

Let $L_e' = L_e \otimes \Dd$, it is of degree $e'$. Using Lemma \ref{equiv lemma}, $\Triv^0(\alpha_N \ | \ \Pp rym_{\PGLn}^{e'})$ parametrises $\Gamma_N$-equivariant universal line bundles on $\Pp rym^{L_e'} \times \Tilde{\Cc}$, with the condition that any pullback to $\Pp rym^{L_e'}$ is  of degree $0$. Two such objects differ by a $\Gamma$-equivariant line bundle of degree $0$ on $\Pp rym^{L_e'}$, and $\Pp ic^0_{\Gamma}(\Pp rym^{L_e'}) = \Pp ic^0(\Pp rym^{L_e'}/\Gamma) \cong \Pp rym^0 $. Hence, $\Triv^0(\alpha_N \ | \ \Pp rym_{\PGLn}^{e'})$ is a $\Pp rym^0$-torsor. 

We construct a map $ \Pp rym^{N} \longrightarrow \Triv^0(\alpha_N \ | \ \Pp rym_{\PGLn}^{e'})$. Let $T \longrightarrow S$ and $L_1 \in \Pp rym^{N}(T)$. We can consider 
    $$\begin{tikzcd}
   \Pp ic^1_{\Tilde{\Cc}/S} \times \Pp ic^{e'}_{\Tilde{\Cc}/S} 
 \arrow[r, "g_{L_1}"] & \Pp ic^0_{\Tilde{\Cc}/S} \times \Pp ic^0(\Pp ic^0_{\Tilde{\Cc}/S} ) \\ 
  (M, \Ll)  \arrow[r] & \big( M \otimes L_1^{-1}, \Oo_{\Pp ic^0_{\Tilde{\Cc}/S} }( [\Ll \otimes L_1^{-e'}] - [0]) \big).
\end{tikzcd}$$



Let $\LL_{L_1} = g_{L_1}^*\Pp_{\Tilde{\Cc}}$. Since $L_e' \in \Pp ic^{e'}_{\Cc/S}(S)$, there is an isomorphism of étale group sheaves over $S$,
\begin{align*}
       \frac{\Pp rym^{L_e'} \times \Pp ic^0_{\Cc/S} }{\Gamma}  
        &\longrightarrow \Pp ic^{e'}_{\Tilde{\Cc}/S} \\
        [(A, B)]  &\longmapsto A \otimes \pi^*B^{-1}.
\end{align*}

Thus we can lift $ \LL_{L_1}$ to a $\Gamma$-equivariant line bundle $ \widetilde{\LL}_{L_1}$ over $\Pp rym^{L_e'} \times \Pp ic^0_{\Cc/S} \times \Pp ic^1_{\Tilde{\Cc}/S} $. Then, the universality of $\LL_{L_1}$ implies that \begin{itemize}[-]
    \item The restriction of $\widetilde{\LL}_{L_1}$ to $\Pp rym^{L_e'} \times \{ \Oo_\Cc \} \times \Pp ic^1_{\Tilde{\Cc}/S} $ is a universal line bundle, using $L_1$ again to embed $\Pp rym^{L_e'} \hookrightarrow \Pp ic^0(\Pp ic^1_{\Tilde{\Cc}/S} ) \simeq \Pp ic^0_{\Tilde{\Cc}/S}$. Note that by construction any of its pullback to $\Pp rym^{L_e'}$ is of degree $0$.
    \item The restriction to $\{A\} \times \Pp ic^0_{\Cc/S} \times \Pp ic^1_{\Tilde{\Cc}/S} $ is a universal line bundle twisted by $A$ in the sense that for all $B \in \Pp ic^0_{\Cc/S}$, the restriction to $\{A\} \times \{ B \} \times \Pp ic^1_{\Tilde{\Cc}/S} $ is isomorphic to $A \otimes (\Nm^{*} B)^{-1}$ over $\Pp ic^1_{\Tilde{\Cc}/S}$ (using that $\pi^*$ and $\Nm$ are dual). 
\end{itemize}

Hence, there exists $N' \in \Pp ic^1_{\Cc/S}(S)$ such that 
$$ \widetilde{\LL}_{L_1}  = p_{13}^* {\Ll}_{L_1}' \otimes p_{23}^* ((\id \times  g_{N'} \circ \Nm)^* \Pp_{\Cc})^{-1} $$

where 
\begin{itemize}[.]
    \item ${\Ll}_{L_1}'$ is a universal line bundle on $\Pp rym^{L_e'} \times \Pp ic^1_{\Tilde{\Cc}/S} $,
    \item $\Pp_{\Cc} \longrightarrow \Pp ic^0_{\Cc/S} \times \Pp ic^0_{\Cc/S}$ is the Poincaré sheaf for $\Pp ic^0_{\Cc/S}$,
    \item $g_{N'} \circ \Nm \ : \ \Pp ic^1_{\Tilde{\Cc}/S} \xrightarrow{\Nm} \Pp ic^1_{\Cc/S} \xrightarrow{\cdot N'^{-1}} \Pp ic^0_{\Cc/S}$. 
\end{itemize}

By construction $((\id \times  g_{N'} \circ \Nm)^* \Pp_{\Cc})^{-1}$ is normalised at $Nm^{-1}(N')$. As ${\LL}_{L_1} $ is normalised at $L_1$, in fact $$ N' = Nm(L_1) = N.$$
Recall the relative group scheme 
$$g_{N}^* \coprod_{\gamma  \in \Gamma} \LL_\gamma^* \longrightarrow \Pp ic^1_{\Cc/S}$$
with a canonical identification of the fibre of $N$ with $\GG_m \times \Gamma$. It acts on $(\id \times g_N)^*\Pp_{\Cc}$, hence on $((\id \times  g_{N} \circ \Nm)^* \Pp_{\Cc})^{-1} $, and $\GG_m$ acts with weight $-1$ (due to the inverse). As $\GG_m$ acts trivially on $\widetilde{\LL}_{L_1}$ (it is simply $\Gamma$-equivariant), ${\Ll}_{L_1}'$ is a $ g_{N}^* \underset{\gamma  \in \Gamma}{\coprod}  \LL_\gamma^*$-equivariant universal line bundle on ${\Pp rym^{L_e'} \times \Pp ic^1_{\Tilde{\Cc}/S}} $, where $\GG_m$ acts with weight $1$.
 
Pulling back ${\Ll}_{L_1}'$ over the curve via $\Tilde{\Cc} \xrightarrow{AJ} \Pp ic^1_{\Tilde{\Cc}/S}$, we obtain a $\Gamma_{N}$-equivariant universal line bundle on $\Pp rym^{L_e'} \times \Tilde{\Cc}$, and thus a $\Gamma$-equivariant trivialisation of $\alpha_{N}$ on $\Pp rym^{L_e'}$ by Lemma \ref{equiv lemma}. Hence, $\Pp rym^{N}$ maps to $\Triv^0(\alpha_N \ | \ \Pp rym_{\PGLn}^{e'})$. Both are $\Pp rym^0$-torsors so they are isomorphic. \end{proof}

\section{A non-archimedean topological mirror symmetry} \label{sectionnaTMS}

Let $S = \Spec \Oo_F$, where $\Oo_F$ is the ring of integers of a local field $F$ and $k$ is the residue field. We assume that $k$ is finite with characteristic $p$ which does not divide $n$. We use the relative construction of the moduli spaces of Higgs bundles over $S$ in \autoref{prelim}. 

\subsection{The set-up for integration} \label{measure}

We refer to \cite[Section 4]{GWZ} for an introduction to $p$-adic integration. In \cite{COW}, the authors constructed a $p$-adic measure in the following situation. Let $\Mm$ be a stack over $S$. Assume : 
    \begin{itemize}[-]
      \item  $ \Mm \longrightarrow S $ is normal of dimension $n$, with smooth atlas $s \ : \ V \longrightarrow \Mm $. 
        \item There is an open stabilizer-free substack $U' \subset \Mm$ whose complements has codimension at least $2$. 
         \item There is an open substack $U \subseteq U'$ and a morphism $\pi \ : \ \Mm \longrightarrow M$  where $M$ is a quasi-projective variety and $\pi_{|U}$ is an isomorphism. Its image is denoted $U$ as well.  
        \item For $x \in M(\Oo_F)^\sharp := M(\Oo_F) \cap U(F) $, there exists a finite extension $L/F$ such that there is a lift of $x$ as an $\Oo_L$-point $x_L \ : \ \Spec \ \Oo_L \longrightarrow \Mm $. 
    \end{itemize}

In this situation, there is a measure $\mu_{\can}$ on $M(\Oo_F)^\sharp$ constructed out of local trivialising top forms. Moreover, the measure does not depend on their choices \cite[Proposition 3.1.2]{COW}.

\begin{lemma} \label{lemma measure}
   The $\Oo_F$-variety $M_{\SLn}^L(\Oo_F)^\sharp$ (resp. $M_{\PGLn}^e(\Oo_F)^\sharp$) satisfies the conditions in \cite{COW} to be equipped with a $p$-adic measure $\mu_{\can, \SLn}$ (resp. $\mu_{\can, \PGLn}$), with respect to morphisms $\Mm_{\SLn}^{L,\rig} \longrightarrow M_{\SLn}^L$ (resp. $\Mm_{\PGLn}^{e} \longrightarrow M_{\PGLn}^e$). 
\end{lemma}

\begin{proof}
 In \cite[Theorem 2.3.1.]{COW}, the authors checked these conditions for $\Mm_{\GLn}^d \longrightarrow M_{\GLn}^d$, and their arguments work for $\SLn$ and $\PGLn$. We summarize it here.
 
\textit{Normality.} 
 As explained in \autoref{section symplectic}, the stack $\Mm_{\SLn}^d \longrightarrow S$ is smooth for $\deg D > 2g - 2$ and normal for $D=K_\Cc$. Rigidification by a finite group preserves smoothness and normality, thus it holds also for $\Mm_{\SLn}^{L,rig} \longrightarrow
 S$. 
 
Taking quotient by a finite abelian group does not affect the cotangent complex nor the atlas of a stack, hence $\Mm_{\PGLn}^e = \Mm_{\SLn}^{L_e,\rig}/\Gamma$  for any $L_e \in \Pp ic^e_{\Cc/S}(S)$ is also smooth (resp. normal) over $S$ for $\deg D > 2g - 2$ (resp. $D=K_\Cc$). 

\textit{Stabilizer free open set of codimension $2$.} We can consider $U' = h_{\SLn}^{-1}(A_{\SLn}^{\textnormal{red}})$, where $A^{\textnormal{red}}$ is the locus of reduced spectral curve. Replacing the dimension of the Hitchin base for $\GLn$ by $$ \dim \ \Aa = \deg \ D \big( \frac{n(n+1)}{2} - 1 \big) + (n-1)(1-g) $$ in the calculation of \cite{COW} leads to the condition 
$$codim \  S_{k, n_1, n_2} \geq 3 \deg D - (2g -2) $$
where $S_{k, n_1, n_2}$ are strata of $\Aa^{\textnormal{red}}$. 
For $D = K$, $g > 1$ or $\deg D > 2g - 2$, $g \geq 1$, the codimension of all strata are greater than $2$, and so is $U'$ by equidimensionality of the Hitchin fibration. 

For $\PGLn$, let $U^{\circ}$ be the open set where $\Gamma$ acts freely. It suffices to consider $ U' \cap U^{\circ}/\Gamma$, who has the same codimension as $U'$.

\textit{Generic isomorphism.}
The open substack $U$ can be chosen to be the open substack of Higgs bundles which are stable over the algebraic closure $\overline{F}$.

\textit{Existence of lifts.}
The last condition follows from \cite[Lemma 21.22]{BLM} or \cite[Theorem A8]{AHLH} (using the GIT construction of $M_{\SLn}^{L}$ ). 
\end{proof}

\subsection{Gerbe functions and \texorpdfstring{$SYZ$}{SYZ}-type duality}  \label{recapsection}
In this section, we convert each $SYZ$-type symmetry in an equality of integrable functions. Let $a \in \Aa(\Oo_F) \cap \Aa^{\textnormal{sm}}(F)$ and $S = \Spec F$. The quotient map $\Phi \ : \ \Pp rym^{0} \longrightarrow \Pp rym^{0}/\Gamma$ is a self-dual isogeny \cite[Proposition 7.10]{GWZ} with kernel $\Gamma$. Hence there is an associated long exact sequence of abelian group : 
$$0 \longrightarrow \Gamma \longrightarrow \Pp rym^{0}(F) \xlongrightarrow{\Phi} \Pp rym^{0}/\Gamma(F) \xlongrightarrow{\beta} H^1(F,\Gamma) \xlongrightarrow{\iota} H^1(F, \Pp rym^{0}) \longrightarrow \dots $$

Recall from \ref{gerbefunction} that for a $n$-torsion $\GG_m$-gerbe $\alpha$ over an abelian variety $A$, the gerbe function $f_{\alpha} \colon A(F) \longrightarrow \CC$ can be expressed in terms of the following Tate pairing,
$$ \langle -,- \rangle_A \ : \  A(F) \times H^1(F, \hat{A}) \longrightarrow \QQ/\ZZ$$
with $f_{\alpha}(x) = \langle -, Triv^0(\alpha \ | \ A ) \rangle_A $

From the first $SYZ$-type symmetry, Proposition \ref{firstSYZ}, since $h^{-1}_{\PGLn}(a) \simeq \Pp rym^{e'}_{\PGLn}$, we have $$ \langle -, Triv^0 (\alpha^{e'} \ | \ \Pp rym^{L \otimes \Dd}) \rangle_{\Pp rym^0} = \langle -, h^{-1}_{\PGLn}(a) \rangle_{\Pp rym^0}  $$ 
on $\Pp rym^0(F)$. 

The second $SYZ$-type symmetry, Proposition \ref{secondSYZ} can also be converted in similar terms. Writing $L \otimes  \Dd = N^{d'} \otimes L_0 $ with $L_0 \in \Pp ic^0(\Cc)$, 
\begin{align} \label{iso pour prymL}
    h^{-1}_{\SLn}(a) &\simeq \Pp rym^{L\otimes \Dd}   \\
    &\simeq (\Pp rym^{N})^{d'} \times^{\Pp rym^0} \Pp rym^{L_0} \nonumber \\
    &\simeq \Triv^0(\alpha_N^{d'} \ | \ \Pp rym_{\PGLn}^{e'})  \times^{\Pp rym^0} \Pp rym^{L_0} \nonumber
\end{align} 

as $\Pp rym^0$-torsors, where we denote by $\times^{\Pp rym^0}$ the product of $\Pp rym^0$-torsor. We used the isomorphism $\Triv^0(\alpha_N^{d'} \ | \ \Pp rym_{\PGLn}^{e'}) \simeq (\Pp rym^{N})^{d'}$ given by the second SYZ symmetry statement (Proposition \ref{secondSYZ}).

Combining three standard long exact sequences, we obtain the following commutative diagram of abelian groups :

\begin{equation} \label{bigdiagram}
    \begin{tikzcd}  
& 1 \arrow[d] & 1 \arrow[d] & 1 \arrow[d] & \dots \arrow[d, "\beta"] \\
  1 \arrow[r]&  \Gamma(F) \arrow[d, "\pi^*"]\arrow[r]& \Pp ic^0_{\Cc/S}(F)\arrow[d, "\pi^*"] \arrow[r, "\cdot n"]&  \Pp ic^0_{\Cc/S}(F)\arrow[d, "\id"]\arrow[r, "\delta"] & H^1(F, \Gamma)  \arrow[d, "\iota"] \\
   1 \arrow[r]& \Pp rym^{0}(F) \arrow[d, "\Phi"]\arrow[r] & \Pp ic^0_{\tilde{\Cc}/S}(F) \arrow[d]\arrow[r, "\Nm"]& \Pp ic^0_{\Cc/S}(F) \arrow[d] \arrow[r, "\Delta"]& H^1(F, \Pp rym^{0} ) \arrow[d] \\
   1 \arrow[r]&  \Pp rym^{0}/\Gamma(F) \arrow[d, "\beta"]\arrow[r, "\id"]& \Pp rym^{0}/\Gamma(F) \arrow[d]\arrow[r]& 1 \arrow[r]& H^1(F, \Pp rym^{0}/\Gamma) \\
   & \dots & \dots & 
\end{tikzcd}
\end{equation}

Since $\Gamma$ is Cartier self-dual, $H^1(F, \Gamma)$ is equipped with a Tate pairing \cite[Corollary I.2.3]{milnearithmdual} : 
$$ \langle - , - \rangle_\Gamma \ : \ H^1(F, \Gamma) \times H^1(F, \Gamma) \longrightarrow \QQ/\ZZ. $$

Moreover, there is an identity for $t \in H^1(F, \Gamma)$ and $y \in \Pp rym^{0}/\Gamma(F)$ \cite[eq. (34)]{GWZ} : 
$$ \langle y, \iota(t) \rangle_{\Pp rym^{0}/\Gamma} = \langle  \beta(y), t \rangle_\Gamma. $$

Since $\Pp rym^{L_0}$ represents $\Delta(L_0)$ and $ \Delta(L_0) = \iota(\delta(L_0))$, 
\begin{equation} \label{equation pairings}
    \langle y, \Pp rym^{L_0} \rangle_{\Pp rym^{0}/\Gamma} = \langle \beta(y), \delta(L_0) \rangle_\Gamma.
\end{equation}

Then, using (\ref{iso pour prymL}), 
\begin{align*}
   \langle y, Triv^0 (\alpha_N^{d'} \ | \ \Pp rym_{\PGLn}^{e'}) \rangle_{\Pp rym^0/\Gamma} + \langle \delta(L_0), \beta(y) \rangle_\Gamma 
   &=  \langle y, \Pp rym^{N^{d'}} \rangle_{\Pp rym^0/\Gamma}  + \langle y, \Pp rym^{L_0} \rangle_{\Pp rym^0/\Gamma} \\
    &=\langle y, \Pp rym^{N^{d'}} \times^{\Pp rym^0} \Pp rym^{L_0} \rangle_{\Pp rym^0/\Gamma} \\
    &= \langle y, h^{-1}_{\SLn}(a) \rangle_{\Pp rym^0/\Gamma}. 
\end{align*} 

We just proved the following lemma.

\begin{lemma} \label{recaplemma}
Let $S = \Spec F$. Let $d, e \in \ZZ$ and $L \in \Pp ic^d_{\Cc/S}(S)$. Consider the Hitchin maps $h_{\SLn} \colon M_{\SLn}^L \longrightarrow \Aa$ and $h_{\PGLn} \colon M_{\PGLn}^e \longrightarrow \Aa$. Let $a \in \Aa^{\textnormal{sm}}(F) \cap \Aa(\Oo_F)$. 
\begin{enumerate}
    \item Let $\alpha$ denote the gerbe ${\Mm^{L}_{\widetilde{\SLn}} \longrightarrow \Mm^{L,\rig}_{\SLn}}$. For $x \in \Pp rym^0(S)$, 
$$\langle x, h^{-1}_{\PGLn}(a) \rangle_{\Pp rym^0} = \langle x, Triv^0 (\alpha^{e'} \ | \ h^{-1}_{\SLn}(a))\rangle_{\Pp rym^0}.  $$ 
    
    \item Let $N \in \Pp ic^1_{\Cc/S}(S)$ and $L_e \in \Pp ic^e_{\Cc/S}(S)$. Let $\alpha_N$ denote the $\Gamma$-equivariant gerbe ${\Mm^{L_e}_{\widetilde{\SLn}} \longrightarrow \Mm^{e}_{\PGLn}}$ from (\ref{pgln gerbe eq}). Set ${L_0 = L \otimes \Dd \otimes N^{-d'} \in \Pp ic^0_{\Cc/S}(S)} $. For $y \in \Pp rym^0/\Gamma(S)$, 
$$ \langle y, h^{-1}_{\SLn}(a) \rangle_{\Pp rym^0/\Gamma} = \langle y, Triv^0 (\alpha_N^{d'} \ | \ h^{-1}_{\PGLn}(a)) \rangle_{\Pp rym^0/\Gamma} + \langle \beta(y), \delta(L_0)  \rangle_\Gamma .$$     
\end{enumerate}
\end{lemma}

As explained in \ref{gerbefunction}, gerbe functions are intrinsic analogues to $\langle - , Triv^0 \rangle $, meaning that gerbe functions are defined without respect to a choice of trivialisation $h_{\SLn}^{-1}(a) \xrightarrow{\sim} \Pp rym^0$ and $h_{\PGLn}^{-1}(a) \xrightarrow{\sim} \Pp rym^0/\Gamma$. 

Since the function $\beta \colon \Pp rym^0/\Gamma(F) \longrightarrow H^1(F, \Gamma)$ is defined on the abelian variety $\Pp rym^0/\Gamma(F)$, and not directly on the Hitchin fibre, we need an intrinsic analogue to $\langle \beta(-), \delta(L_0)  \rangle_\Gamma $. 

Let $e \in \ZZ$ and $L_e \in \Pp ic^e_{\Cc/S}(S)$. There is a $\Gamma$-quotient map 
$$ M^{L_e}_{\SLn} \longrightarrow M^{e}_{\PGLn}$$
which is a free quotient above the chosen open set $U \subset M^{e}_{\PGLn}(F)$ in Lemma \ref{lemma measure} defining $M_{\PGLn}^e(\Oo_F)^\sharp$. Hence, for $y \colon \Spec \Oo_F \longrightarrow M_{\PGLn}^e$ in $M_{\PGLn}^e(\Oo_F)^\sharp$, there is an induced $\Gamma$-torsor $T_y = \Spec F \times_{M^{e}_{\PGLn}} M^{L_e}_{\SLn} $  over $\Spec F$. 

\begin{definition}[The function $\beta_{L_e}$] \label{def beta Le}
    Let $e \in \ZZ$ and $L_e \in \Pp ic^e(\Cc)$. We define 
    $$ \beta_{L_e} \colon M_{\PGLn}^e(\Oo_F)^\sharp \longrightarrow H^1(F, \Gamma),  $$
    sending $y \colon \Spec \Oo_F \longrightarrow M_{\PGLn}^e$ to $[T_y]$. 
\end{definition}

For a different choice of line bundle $L_e' \in \Pp ic^e_{\Cc/S}(S)$, 
the function $\beta_{L_e'}$ differs from $\beta_{L_e}$ by multiplication by $\delta(L_e^{-1} L_e') \in H^1(F, \Gamma)$. 

A trivialisation of the $\PGLn$-Hitchin fibre is given by $x \in \Pp rym^{e'}_{\PGLn}(F)$ and denoted $h_x$. Such $x$ lifts to $\tilde{x} \in \Pp rym^{L_x}_{\SLn}$ for a certain line bundle $L_x$ of degree $e$. Then, $\beta \circ h_x$ agrees with $\beta_{L_x}$, both sending $x$ to $[0] \in H^1(F, \Gamma)$.

In turn, for $x \in \Pp rym^{e'}_{\PGLn}(F)$,
   \begin{align} \label{eqbeta}
       \exp (2i\pi \cdot \langle \beta_{L_e}(x), \delta(L_0) \rangle_{\Gamma}) 
       &=  c \cdot \exp (2i\pi \cdot \langle \beta(h(x)), \delta(L_0) \rangle_{\Gamma}) \\
       &=c \cdot \exp (2i\pi \cdot\langle h(x), \Pp rym^{L_0} \rangle_{\Pp rym^0}) \nonumber
   \end{align}
where $c$ is a constant depending on the choice of trivialisation as in \cite[Lemma 6.7]{GWZ}. The last equality is given by (\ref{equation pairings}). 
Note that both sides of (\ref{eqbeta}) are constantly $1$ when $L_0$ has a $\mathit{n}$th root, since $\delta(L_0) = [0]$.



\subsection{The non-archimedean topological mirror symmetry}
Before stating the main theorem, we recall the main objects that come into play. We fix a rank $n \in \NN$. We work over a base scheme $S = \Spec \Oo_F$, where $\Oo_F$ is the ring of integers of a local field $F$ of finite residue field $k_F$, chosen so that $\charac k_F \nmid n$ and $\mu_{n, S}$ is constant. We fix a relative curve $\Cc \longrightarrow S$ of genus $g$ and a relative divisor $D$ such that $\deg D > 2g - 2$ or $D=K_\Cc$. We denote $\Gamma = \Pp ic^0_{\Cc/S}[n]$, which we assume to be constant.

For the $\SLn$ side, we fix a degree $d \in \ZZ $ and a line bundle $L \in \Pp ic^d_{\Cc/S}(S)$. We consider the moduli space $M_{\SLn}^L$ of semi-stable Higgs bundles of determinant $L$ and traceless Higgs field, relative over $S$. We explained in \autoref{slngerbe} that it is the good moduli space of a stack $\Mm^{L,\rig}_{\SLn}$ endowed with a $n$-torsion $\GG_m$-gerbe $\alpha \colon \Mm^{L}_{\widetilde{\SLn}} \longrightarrow \Mm^{L,\rig}_{\SLn}$.

For the $\PGLn$ side, we fix a degree $e \in \ZZ$ and line bundles $L_e \in \Pp ic^e_{\Cc/S}(S)$ and $N \in \Pp ic^1_{\Cc/S}(S)$. The space $M^e_{\PGLn}$ is the good moduli space of a stack $\Mm^e_{\PGLn}$ which admits a $n$-torsion $\GG_m$-gerbe $\alpha_N \colon \Mm^{L_e}_{\widetilde{\SLn}} \longrightarrow \Mm^{L_e,\rig}_{\SLn}/\Gamma \simeq \Mm^e_{\PGLn}$. It depends on $N$, see \autoref{pglngerbe}. We have introduced complex valued functions associated to gerbes in \ref{gerbefunction}. 

We denote $\Dd = \overset{n-1}{\underset{i=0}{\bigotimes}} \Oo_\Cc(-iD)$, $d' = d + \deg \Dd$, $e' = e + \deg \Dd$ and ${L_0 = L \otimes \Dd \otimes N^{-d'} \in \Pp ic^0_{\Cc/S}(S)}$. Then $\delta(L_0) \in H^1(F, \Gamma)$ measure its class in $\Pp ic^0_{\Cc/S}(\Oo_F)/(\Pp ic^0_{\Cc/S}(\Oo_F))^{\times n}$ (see  \autoref{recapsection}). We also introduced a function $\beta_{L_e} \colon M_{\PGLn}^e(\Oo_F)^\sharp \longrightarrow H^1(F, \Gamma)$ in Definition \ref{def beta Le}. 

Finally, for a good moduli space map $\Mm \longrightarrow M$ which is an isomorphism when restricted to an open set $U$, we write
$$ M(\Oo_F)^\sharp := M(\Oo_F) \cap U(F), $$
and by Lemma \ref{lemma measure} there are canonical $p$-adic measures $\mu_{\can, \SLn}$ and $\mu_{\can, \PGLn}$ on $M_{\SLn}^L(\Oo_F)^\sharp$ and $M_{\PGLn}^e(\Oo_F)^\sharp$.

\begin{theorem}[Non-archimedean topological mirror symmetry] \label{mainthm}
With the above notations, the following equality holds : 
     $$ \int_{M_{\SLn}^L(\Oo_F)^\sharp} f_{\alpha}^{e'} \ \mu_{\can, \SLn} = \int_{M_{\PGLn}^e(\Oo_F)^\sharp} f_{\alpha_{N}}^{d'} \cdot exp(2i\pi \cdot \langle \beta_{L_e}(-), \delta(L_0)\rangle_\Gamma )  \ \mu_{\can, \PGLn}. $$
\end{theorem}

\begin{remark}
Degrees $e$ and $e'$ on the $\PGLn$ side are only well-defined modulo $n$ but the gerbe function $f_\alpha$ is $n$-torsion (see \autoref{slngerbe}) so that $f_{\alpha}^{e'}$ is well-defined.   
\end{remark}

\begin{proof}
    We use the same proof strategy as in \cite[Theorem 6.12]{GWZ} and \cite[Thm 5.0.2]{COW}. 

Let $\Aa^{\textnormal{red}}_S$ denote the locus of reduced spectral curves. Let $\Pp rym^0_{\SLn} \longrightarrow \Aa^{\textnormal{red}}_S $ and $\Pp rym^0_{\PGLn} \longrightarrow \Aa_S^{\textnormal{red}}$ be the relative group schemes acting faithfully and transitively on $ \Mm_{\SLn}' = M_{\SLn}^L \times_\Aa \Aa^{\textnormal{red}}$ and $\Mm_{\PGLn}' = M_{\PGLn}^e \times_\Aa \Aa^{\textnormal{red}}$ \cite[Proposition 4.3.3]{ngo}. 
 
    There exists a global, translation-invariant trivialising volume form $\omega_{\Pp, \PGLn}$ for $ \Omega^{top}_{\Pp rym^0_{\PGLn}/\Aa^{\textnormal{red}}}$. Using the isogeny 
    $$ \Phi \colon \Pp rym^0_{\SLn} \longrightarrow \Pp rym^0_{\PGLn} $$
    the pullback $\omega_{\Pp, \SLn} = \Phi^* \omega_{\Pp, \PGLn}$ gives a global, translation-invariant, trivialising relative volume form for $\Pp rym^0_{\SLn} \longrightarrow \Aa_S^{\textnormal{red}}$.

    Using \cite[Lemma 6.13]{GWZ}, those relative forms induce trivialising relative forms for the torsors $\Mm_{\SLn}'$ and $\Mm_{\PGLn}'$, denoted $\omega_{h, \SLn}$ and $\omega_{h, \PGLn}$. 
    
    There are isomorphisms of line bundles :
    $$ \Omega^{top}_{\Mm_{\SLn}'/S} \simeq \Omega^{top}_{\Mm_{\SLn}'/\Aa^{\textnormal{red}}} \otimes \Omega^{top}_{\Aa^{\textnormal{red}}/S} $$
    and 
    $$ \Omega^{top}_{\Mm_{\PGLn}'/S} \simeq \Omega^{top}_{\Mm_{\PGLn}'/\Aa^{\textnormal{red}}} \otimes \Omega^{top}_{\Aa^{\textnormal{red}}/S} .$$
    
    Moreover, $\Aa$ being an affine space, there exists a global volume form $\omega_{\Aa}$ on $\Aa$. We obtain global volume forms on $\Mm_{\SLn}'$ and $\Mm_{\PGLn}'$, 
    $$ \omega_{\SLn}' = h_{\SLn}^* \omega_\Aa \wedge \omega_{h, \SLn} $$
    $$ \omega_{\PGLn}' = h_{\PGLn}^* \omega_\Aa \wedge \omega_{h, \PGLn}. $$

    Since $\Mm_{\SLn}^L \setminus \Mm_{\SLn}'$ is of codimension at least $2$, by Hartog's extension theorem $\omega_{\SLn}'$ extends to a global trivialising form $\omega_{\SLn}$ on $\Mm_{\SLn}^L(\Oo_F)^{\sharp}$. The same apply to construct $\omega_{\PGLn}$. Moreover, $\Mm_{\PGLn}' \subset U' \cap U^{\circ}/\Gamma$ from Lemma \ref{measure} since $\Mm_{\PGLn}'$ is represented. Hence, $\mu_{\can, \SLn}$ (resp. $\mu_{\can, \PGLn}$) is given by integrating global volume forms $\omega_{\SLn}$ (resp. $\omega_{\PGLn}$) on $M_{\SLn}^L(\Oo_F)^\sharp$ (resp. $M_{\PGLn}^e(\Oo_F)^\sharp$).

    Then, the volume of $ M_{\SLn}^L \times_\Aa (\Aa \setminus \Aa^{\textnormal{sm}})$ (resp. $M_{\PGLn}^e \times_\Aa ((\Aa \setminus \Aa^{\textnormal{sm}})$) is zero, since it is a closed subscheme of positive codimension \cite[Proposition 4.4]{GWZ}. Since the Hitchin map is proper over $\Aa$ and using the Fubini theorem \cite[Proposition 4.1]{GWZ}, we reduce to the fibrewise equality 
    $$ \int_{h_{\SLn}^{-1}(a)(F)} f_{\alpha}^{e'} \ |\omega_{\SLn, a} | = \int_{h_{\PGLn}^{-1}(a)(F)} f_{\alpha_{N}}^{d'} \cdot exp(2i\pi \cdot \langle \beta_{L_e}(-), \delta(L_0)\rangle_\Gamma )  \ |\omega_{\PGLn, a} |$$
    for all $a \in \Aa(\Oo_F) \cap \Aa^{\textnormal{sm}}(F)$, where $\omega_{\SLn, a}$ (resp. $\omega_{\PGLn, a}$) is the restriction of $\omega_{\SLn}$ (resp.  $\omega_{\PGLn}$) to the fibre over $a$. 
    
    Pick $a \in \Aa(\Oo_F) \cap \Aa^{\textnormal{sm}}(F)$. There is an associated smooth spectral curve $\Tilde{\Cc}$, for which we use the notation $\Pp rym^0$, $\Pp rym^0/\Gamma$ as introduced in \autoref{pryms}. 
    

    Then, there are four possible cases. 

    \begin{itemize}[-]
        \item If both fibres are empty, both sides are zero. 
        \item If both fibres have a rational point, then they are trivial $\Pp rym^0$-torsor and  $\Pp rym^0/\Gamma$-torsor, so we can pick two trivialisations. By Lemma \ref{recaplemma} and passing to the exponential, the integrands can be expressed in terms of Tate pairing against the dual fibre, $\exp^{2 i \pi \cdot \langle -, h^{-1}_{\SLn}(a)(F) \rangle} $ (resp. $\exp^{2 i \pi \cdot \langle -, h^{-1}_{\SLn}(a)(F) \rangle} $), so they are identically $\id$.

        It suffices to prove that the $p$-adic volume of $\Pp rym^0$ and $\Pp rym^0/\Gamma$ is the same. It is the content of \cite[Lemma 6.15]{GWZ}. The proof relies on the Fubini theorem \cite[Proposition 4.4]{GWZ}, the construction of the translation-invariant form $\omega_{\SLn}$ by pullback along $\Phi$ and the fact that the quotient map $\Phi \ : \ \Pp rym^0 \longrightarrow \Pp rym^0/\Gamma$ is a self-dual isogeny, so that \cite[Proposition 3.16]{GWZ} holds : 
        $$\bigg|  \frac{\Pp rym^0/\Gamma(F)}{\Phi(\Pp rym^0)(F)} \bigg| = \mid \ker \Phi (F) \mid.   $$

        \item If $h_{\SLn}^{-1}(a)(F) = \emptyset$, and $h_{\PGLn}^{-1}(a)(F) \neq \emptyset$, we pick a trivialisation $h \ : \ h_{\PGLn}^{-1}(a) \xrightarrow{\sim} \Pp rym^0/\Gamma$. Using Lemma \ref{recaplemma}, the integrand on the right is $\exp^{2 i \pi \cdot \langle -, h_{\SLn}^{-1}(a) \rangle}$ up to a constant due to the choice of trivialisation. Since $h_{\SLn}^{-1}(a)(F)$ is empty, it is a non-trivial $\Pp rym^0$-torsor. The integrand is then a non-trivial character of $ \Pp rym^0/\Gamma$, so the integral is $0$.  

        \item If $h_{\SLn}^{-1}(a)(F) \neq \emptyset$, and $h_{\PGLn}^{-1}(a)(F) = \emptyset$, we use a symmetric argument.     
    \end{itemize}
\end{proof}

\subsection{A refinement}





Recall the map $\beta_{L_e}$ from Definition \ref{def beta Le} : we fix $e \in \ZZ$ and $L_e \in \Pp ic^e_{\Cc/S}(S)$, then
    $$ \beta_{L_e} \colon M_{\PGLn}^e(\Oo_F)^\sharp \longrightarrow H^1(F, \Gamma),  $$
    $$ y \longmapsto [T_y]$$
    
    where $T_y = \Spec F \times_{M^{e}_{\PGLn}} M^{L_e}_{\SLn} $ is the induced $\Gamma$-torsor over $\Spec F$. 

As explained in \cite[Proposition 3.6]{GWZ}, for a choice of primitive $\mathit{n}$th root of unity, there is a decomposition corresponding to ramified and unramified part $$H^1(F, \Gamma) \simeq H^1_{\textnormal{ur}}(F, \Gamma) \oplus \Gamma \simeq \Gamma \oplus \Gamma$$
where $H^1_{\textnormal{ur}}(F, \Gamma) = H^1_{\textnormal{ét}}(k_F, \Gamma)$ is the group of torsors which lifts to $\Spec \Oo_F$. 

The lift of $y_{\Oo_F}$ to $\Mm^e_{\PGLn}$ has a unique closed point. Let $\gamma$ be the generator of its automorphism group. Then, $\gamma$ is the projection of $[T_y]$ to the ramified part. The construction is detailed in \cite[Construction 4.15]{GWZ} and \cite[Construction 3.1.2]{LW}. Composing $\beta_{L_e}$ with this second projection $H^1(F, \Gamma) \longrightarrow \Gamma$, we obtain a specialisation map :
$$ M_{\PGLn}^e(\Oo_F)^\sharp \xlongrightarrow{s} \Gamma. $$

It does not depend on $L_e$, since a different choice gives a multiplication by $\delta(L_e^{-1} L_e') \in H^1_{\textnormal{ur}}(F, \Gamma)$, which does not affect the ramified part.

\begin{remark}[Weil pairing] \label{rmk weil pairing}
  The group $\Gamma$ is a self-dual abelian variety as kernel of the self-dual isogeny $[n]$. Thus, there is a Weil pairing $(- , - )_\Gamma \colon \Gamma \times \Gamma \longrightarrow \mu_n(\CC)$. It provides a canonical identification $\omega \ : \ \Gamma \longrightarrow \hat{\Gamma}$. For $\gamma \in \Gamma$, we denote $\kappa = \omega(\gamma)$ the corresponding character.
\end{remark}

 We fix an identification $H^1_{\textnormal{ur}}(F, \Gamma) \simeq  \Gamma$, which is given by a choice of uniformiser of $F$ \cite[Construction 3.5]{GWZ}. Recall that $\delta$ denotes the connecting morphism $\delta \colon \Pp ic^0(\Oo_F) \longrightarrow H^1_{\textnormal{ur}}(F, \Gamma)$ in (\ref{bigdiagram}). 

\begin{theorem}[Refinement]\label{fourier}
  Let $d, e \in \ZZ$. Let $L = N^{\otimes d}$ and for $\nu \in \Gamma$, let $L_{\nu} \in \Pp ic^0(\Cc)(\Oo_F)$ such that $\delta(L_\nu) = \nu$. Then, for $\kappa \in \hat{\Gamma}$ and $\omega (\gamma) = \kappa$, 
  $$ \frac{1}{ |\Gamma|} \sum_{\nu \in \Gamma} \Big( \int_{M_{\SLn}^{N^d \otimes L_\nu} (\Oo_F)^\sharp} f_{\alpha}^{e'} \ \mu_{\can} \Big) \kappa(\nu) = \int_{M_{\PGLn}^e(\Oo_F)^\sharp \cap s^{-1}(\gamma)}  f_{\alpha_{N}}^{d'}  \ \mu_{\can}.  $$
\end{theorem}

\begin{proof}
 The pairing $ \langle \beta_{L_e}(-), \delta(L_0)\rangle_\Gamma$ is the Tate pairing
$$ \langle - , - \rangle_\Gamma \ : \ H^1(F, \Gamma) \times H^1(F, \Gamma) \longrightarrow \QQ/\ZZ. $$

Its behaviour with respect to the decomposition $H^1(F, \Gamma) \simeq H^1_{\textnormal{ur}}(F, \Gamma) \oplus \Gamma$ is described in \cite[Lemma 3.7]{GWZ}. The first factors jointly pairs to zero, as well as the second factors. Then, its restriction to $ H^1_{\textnormal{ur}}(F, \Gamma) \times \Gamma $ and $ \Gamma \times H^1_{\textnormal{ur}}(F, \Gamma) $ is given by the Weil pairing $(-,-)_\Gamma$ (see the proof of \cite[Lemma 7.25]{GWZ}). 

Since $L_{\nu} \in \Pp ic^0(\Cc)(\Oo_F)$, $\delta(L_\nu)$ lies in the unramified part $H^1_{\textnormal{ur}}(F, \Gamma)$. It pairs non-trivially only with torsors in the ramified part. Hence, for $y \in M_{\PGLn}^e(\Oo_F)^\sharp$, 
\begin{align*}
    \langle \beta_{L_e}(y), \delta(L_\nu)\rangle_\Gamma &= \langle \ (0, s(y)), (\nu, 0) \ \rangle_\Gamma \\
    &= (s(y), \nu)_\Gamma .
\end{align*}


For $\gamma' \in \Gamma$, we denote $\kappa' := \omega(\gamma')$ (Remark \ref{rmk weil pairing}). Then,

\begin{align*}
    \sum_{\nu \in \Gamma} \Big( \int_{M_{\SLn}^{N^d \otimes L_\nu} (\Oo_F)^\sharp} f_{\alpha}^{e'} \ \mu_{\can} \Big) \kappa(\nu) 
    &= \sum_{\nu \in \Gamma} \Big( \int_{M_{\PGLn}^e(\Oo_F)^\sharp} f_{\alpha_{N}}^{d'} \cdot  exp(2i\pi \cdot \langle \beta_{L_e}(-), \delta(L_\nu)\rangle_\Gamma ) \ \mu_{\can} \Big) \kappa(\nu) \\
    &= \sum_{\nu \in \Gamma} \sum_{\gamma' \in \Gamma} \Big(  \int_{M_{\PGLn}^e(\Oo_F)^\sharp \cap s^{-1}(\gamma')}  f_{\alpha_{N}}^{d'} \cdot ( \gamma', \nu )_\Gamma \ \mu_{\can} \Big) \kappa(\nu)  \\
   &= \sum_{\gamma' \in \Gamma}  \Big( \int_{M_{\PGLn}^e(\Oo_F)^\sharp \cap s^{-1}(\gamma')}  f_{\alpha_{N}}^{d'}  \ \mu_{\can} \Big) \sum_{\nu \in \Gamma} \kappa \kappa'^{-1}(\nu)  \\
    &= |\Gamma| \int_{M_{\PGLn}^e(\Oo_F)^\sharp \cap s^{-1}(\gamma)}  f_{\alpha_{N}}^{d'}  \ \mu_{\can} 
\end{align*}

where we used \autoref{mainthm} for the first equality. Note that $\underset{\nu \in \Gamma}{\sum} \kappa \kappa'^{-1}(\nu) = |\Gamma|$ if $\kappa = \kappa'$ and $0$ otherwise.   
\end{proof}

\section{Interpretation of \texorpdfstring{$p$}{p}-adic integrals}

\label{section corollaries}



In this section, we are interested in interpreting the $p$-adic integrals of \autoref{mainthm} in terms of cohomological invariants. 

\begin{remark}[How to apply \autoref{mainthm} ?] \label{spread out rmk}
   We follow the reasoning in \cite[Remark 3.2.2]{COW}. Given a pair $M_{\SLn}^L, M_{\PGLn}^e$ over $\Spec \CC$, we find a spreading out, i.e. a finitely generated $\ZZ$-algebra $R$ such that the curve $\Cc$ and line bundles $D, L, N$ are defined over $B = \Spec R$, such that $\mu_{n, B}$ and $ \Gamma_B$ are constant. Then, $M_{\SLn}^L \longrightarrow B$ and $M_{\PGLn}^e \longrightarrow B$ are geometrically normal so there is an open $B' \subset B$ whose fibres are normal. It gives infinitely many prime characteristic $p$ for which we can consider a ring of $p$-adic numbers $\Oo_F$ satisfying the assumptions of \autoref{sectionnaTMS}. 
\end{remark}


\begin{notation} \label{notation frob}
    Fixing a prime number $\ell$, for a finite field $k$ of characteristic $p \neq \ell$ we consider the Frobenius element $\Fr$ of $Gal(\overline{k}/k)$ acting on $\ell$-adic complexes over  ${M_{\SLn, k}^L}$ and ${M_{\PGLn, k}^e}$. For a bounded complex $H^\bullet \colon \ldots \longrightarrow H^{i-1} \longrightarrow H^i \longrightarrow H^{i+1} \longrightarrow \dots $ of vector spaces on which the Frobenius acts, the expression 
    $$ \tr(\Fr \ | \ H)$$
    denote the alternated sum $\underset{i \in \ZZ}{\sum} (-1)^i \tr(\Fr \ | \ H^i ) $. 
\end{notation}

\subsection{Completing the coprime case}

The topological mirror symmetry theorem \cite[Theorem 7.21]{GWZ} holds for $\deg D$ even. This restriction avoids the case where $d \not\equiv d' \mod n$. Our main theorem extends this result to the odd degree $D$ case.  

Let $S = \Spec \CC$. Fix a smooth projective curve $\Cc$ of genus $g$ over $\CC$ and line bundles $D, N$ on $\Cc/\CC$ of degree $\deg D$ and $1$ respectively. As before, assume $\deg D > 2g- 2$ or $D \sim K_\Cc$.

\begin{corollary}[Topological mirror symmetry - coprime case] \label{corTMScoprime}
Let $(d,n)=(e,n)=1$. Let $d \in \Pp ic^d(\Cc)$ and $d' = d - \deg D \cdot \frac{n(n-1)}{2}$. We have the following equality of stringy twisted $E$-polynomials (Definition \ref{defi stringy}):
$$ E(M_{\SLn}^{L}) = E_{\textnormal{st}, \alpha_N^{d'}}(M_{\PGLn}^{e}). $$
Let $q$ be the multiplicative inverse of $d$ modulo $n$. For any $\gamma \in \Gamma$ and $L_e \in \Pp ic^e_{\Cc/\CC}(\CC)$, there is a refinement:
$$ E(M_{\SLn}^{L})_{\kappa} = E(M_{\SLn}^{L_e, \gamma})_{\kappa^{-qd'}} $$
where $\kappa = \omega(\gamma)$ as in Remark \ref{rmk weil pairing}.
\end{corollary}

\begin{proof}
In the coprime case, $\Mm_{\SLn,\rig}^{L}$ is a smooth scheme and $\Mm_{\PGLn}^e$ is a smooth Deligne-Mumford stack equipped with the gerbe $\alpha_N^{d'} \colon \Mm^{L_e}_{\SLn} \rightarrow \Mm_{\PGLn}^e$. Since $\CC$ is algebraically closed, $L \otimes \Dd \otimes N^{- d'}$ has a $\mathit{n}$th root $L_0^{\frac{1}{n}}$. 

As explained in Remark \ref{spread out rmk}, is possible to spread out $\Cc, D, L, N, L_0^{\frac{1}{n}}, L_e$ over a finitely generated $\ZZ$-algebra $R$. 
Then, $\Mm_{\SLn, R}^{L,\rig}$ and $\Mm_{\PGLn, R}^e$ form an abstract dual Hitchin system in the sense of \cite[Definition 6.9]{GWZ}.

 Since the gerbe $\alpha \ : \ \Mm^{L}_{\SLn} \longrightarrow \Mm^{L,\rig}_{\SLn} \simeq M_{\SLn}^L$ is over the good moduli space, we have $f_{\alpha} \equiv 1$, as explained in the proof of \cite[Lemma 7.19]{GWZ}. Indeed, for $x \in M_{\SLn}^L(\Oo_F)^\sharp$, $f_{\alpha_{L}}(x) = \exp(2 \pi i \cdot inv(\tilde{x}^* \alpha))$ (Construction \ref{gerbe function construction}). As $\Mm^{L,\rig}_{\SLn} \simeq M_{\SLn}^L$, the lift $\tilde{x}$ extends to the $\Oo_F$-point $x$, and the vanishing of Brauer groups over local fields implies $inv(x^* \alpha) = 0$. 
    
  Thus, \autoref{mainthm} reads : 
 $$ \int_{M_{\SLn}^L(\Oo_F)^\sharp} 1 \ \mu_{\can, \SLn} = \int_{M_{\PGLn}^e(\Oo_F)^\sharp} f_{\alpha_{N}}^{d'}  \ \mu_{\can, \PGLn} .$$

    By \cite[Theorem 6.11]{GWZ}, $M_{\SLn, R}^{L}$ and $\Mm_{\PGLn, R}^e$ have the same stringy twisted $E$-polynomial (with respect to $\alpha_N^{d'}$ for the $\PGLn$-side). Since  $M_{\SLn}^{L}$ is smooth, its stringy invariants corresponds to the classical ones.
    
    The reasoning uses an orbifold formula \cite[Theorem 4.16, Theorem 6.12]{GWZ} relating $p$-adic integrals to finite fields counts. By Grothendieck-Lefschetz theorem, those counts are traces of Frobenius in $\ell$-adic cohomology, which recover $E$-polynomials by means of $p$-adic Hodge theory.

    For the refinement, the proof of \cite[Theorem 7.24]{GWZ} applies. In the coprime case, the gerbe corresponds to a twist by a $\Gamma$-equivariant trivial local system, whose equivariant structure corresponds to the character $\kappa^{-q}$ as explained in \cite[Appendix A]{LW} based on \cite[Proposition 8.1]{HT}. 
\end{proof}

\begin{remark} \label{inaccuracy}
The corollary above allows us to detect an inaccuracy in \cite[Theorem 3.2]{MSendoscopicHTcoprime} for the choice of the character, since the $\kappa$-part and $\kappa^{-qd'}$-isotypical part can be strictly different when $d \not\equiv d' \mod n$. It happens only when $\deg D$ is odd. According to the authors, it is due to a certain $G$-equivariant local system appearing in the vanishing cycles argument, where $G$ is a Galois group of a Galois cover of $\Cc$ related to $\gamma$-fixed locus, as in Proposition \ref{prop link IH}. 

The parity assumption for $D$ was also missing in the non-coprime case in \cite[Theorem 0.2]{MSnoncoprime}. The authors recently uploaded a corrected version on arXiv of both articles \cite{MSarxivcoprime, MSarxivnoncoprime}. 
\end{remark}

\begin{remark}
As computational evidence, the computation in prime ranks in \cite[Proposition 10.1 and 8.2]{HT} can be adapted to the meromorphic case $ D> K_\Cc$, and $d'$ appears in equations (10.2) and (10.3) for the $\kappa$-part on the $\SLn$ side, while the computation for the $\kappa$-part on the $\PGLn$ side is insensitive to $\deg D$. Taking the $d'$th power on the $\PGLn$ side therefore rebalances both sides.
\end{remark}


\subsection{Link to intersection cohomology when \texorpdfstring{$D > K$}{D>K}}

In \cite{MSnoncoprime}, the authors proved the existence of the following isomorphism when $\deg D> 2g - 2$ is even, $S = \Spec \ \CC$ and $L \in \Pp ic_{\Cc/S}(S)$, for $\kappa = \omega(\gamma)$:
\begin{equation} \label{MSeq}
    (R{h_{\SLn, *}} IC _{M_{\SLn}^{L}})_{\kappa} \xlongrightarrow{\sim} i_{\gamma, *} (R{h_\gamma, *} IC_{M_{\SLn}^{L, \gamma}})_{\kappa} [-F_\gamma]
\end{equation}
in the bounded derived category of $\Aa^\gamma = h_{\SLn}(M_{\SLn}^{L, \gamma})$, where $i_\gamma \colon \Aa^\gamma \hookrightarrow \Aa $ denotes the inclusion and $h_\gamma$ the restriction of the Hitchin map to $M_{\SLn}^{L, \gamma}$. This isomorphism allows us to prove the following proposition. 


\begin{proposition} \label{prop link IH}
    Let $n$ be prime. Let $\Cc$, $D$, $L$, $N$ be defined over $\CC$, such that $\deg D> 2g-2$ even and $L= N^{\otimes d}$. There exists a spreading out over $B = \Spec R$ (as in Remark \ref{spread out rmk}) and a dense open $\tilde{B} \longrightarrow B$ such that for any $\Spec k_F \longrightarrow \tilde{B}$, residue field of a ring $\Oo_F$,
    $$ \int_{M_{\SLn}^L(\Oo_F)^\sharp} f_{\alpha} \ \mu_{\can} = q^{- \dim M_{\SLn}^L} \tr ( \Fr \ | \ IH_c(M_{\SLn, k_F}^L, \underline{\QQ}_\ell))$$
    using Notation \ref{notation frob} for the right hand side, with $q = |k_F|$. 
\end{proposition}

\begin{proof}
   We can assume that $L = \Oo_\Cc$. Indeed, we assume $L = N^{\otimes d}$ and since $n$ is prime, the degree $0$ case is the only non-coprime case, and the equality is already known for the coprime case (where intersection cohomology is the same as singular cohomology). We use $M^0_{\SLn}$ to denote $M^{\Oo_\Cc}_{\SLn}$. 
    
    There are three steps for this proof, summarised in the square below. Step $1$ uses our main \autoref{mainthm}. Step $2$ combines the explicit description of $\gamma$-fixed loci in terms of Prym varieties and the $\chi$-independence results for $\GLn$ by Maulik and Shen \cite{MS23}. Step $3$ uses the topological mirror symmetry equality (\ref{MSeq}) from \cite{MSnoncoprime}. 

$$    \begin{tikzcd}[column sep=huge] 
         \int_{M_{\SLn}^0(\Oo_F)^\sharp} f_{\alpha} \ \mu_{\can} \arrow[r, equal, "\text{Step 1 : \autoref{mainthm}}"] &  \int_{M_{\PGLn}^1(\Oo_F)^\sharp} 1 \ \mu_{\can} \arrow[d, equal, "\text{Step 2 : \cite{GWZ}, \cite{MS23}}"] \\
         q^{- \dim M_{\SLn}^0} \tr ( \Fr \ | \ IH_c(M_{\SLn, k_F}^0, \underline{\QQ}_\ell)) \arrow[r, equal, "\text{Step 3 : \cite{MSnoncoprime}}"] &  q^{- \dim M_{\PGLn}^0} \underset{\gamma \in \Gamma}{\sum}  \tr (\Fr \ | \ IH_c^{* -F_\gamma}(M_{\SLn,k_F}^{0,\gamma}, \underline{\QQ}_\ell)_\kappa)
    \end{tikzcd}$$
    
    As in Remark \ref{spread out rmk}, we can find a spreading out of $M_{\SLn}^0$ over $B = \Spec R$. 
    
    \textit{Step 1.} We apply \autoref{mainthm} for triples $(F, \Oo_F, k_F)$ equipped with a map $R \rightarrow \Oo_F$. Since $d' \equiv d \equiv 0 \mod n$ ($\deg D$ is even),
    $$\int_{M_{\SLn}^{0}(\Oo_F)^\sharp} f_{\alpha} \ \mu_{\can} =   \int_{M_{\PGLn}^1(\Oo_F)^\sharp} 1 \ \mu_{\can}. $$
    
    \textit{Step 2.} Using \cite[Theorem 4.16]{GWZ}, we get 
    \begin{align} \label{degree 1 pgln coho}
        \int_{M_{\PGLn}^1(\Oo_F)^\sharp} 1 \ \mu_{\can} &=  q^{- \dim M_{\PGLn}^1} \cdot \tr (\Fr \ | \ H_c^{orb} (M_{\PGLn, k_F}^1, \underline{\QQ}_\ell)) \nonumber \\
        &= q^{- \dim M_{\PGLn}^1} \cdot \bigg( \tr (\Fr \ | \ H_c(M_{\PGLn, k_F}^1, \underline{\QQ}_\ell)) + \sum_{\gamma \neq 1} \tr (\Fr \ | \ H_c^{* - F_\gamma} (M_{\SLn, k_F}^{L_1,\gamma}, \underline{\QQ}_\ell)^\Gamma ) ) \bigg)
    \end{align}  
    for any line bundle $L_1 \in \Pp ic^1_{\Cc/k_F}(k_F)$. 
    
\begin{claim}[First term in (\ref{degree 1 pgln coho})] \label{claim first term}
There exists a dense open set $B' \hookrightarrow
B$ for which at any closed point $\Spec k \rightarrow B'$,
     $$\tr (\Fr \ | \ H_c(M_{\PGLn, k}^1, \underline{\QQ}_\ell)) = \tr (\Fr \ | \ H_c(M_{\PGLn, k}^0, \underline{\QQ}_\ell)).$$
\end{claim}

\begin{proof}[Proof of Claim \ref{claim first term}]
      Arguing as in \cite[Proposition 6.4]{GWZ24}, we use the $\chi$-independence result for $\GLn$-Higgs bundles by Maulik and Shen \cite[Theorem 0.1]{MS23}. 

    There, it is deduced from \cite[Theorem 0.4]{MS23}, which can be rephrased as an isomorphism in $D^b_{c}(\Aa_{\CC}, \QQ_\ell)$, for $d\in \ZZ$, 
         $$ \bigoplus_{i=0}^{2g'} IC(\wedge^i R^1 \pi_* \QQ_\ell) \xlongrightarrow{\sim} {Rh_{\GLn}}_* {IC_{M_{\GLn}^d, \CC}}  $$
    where $\pi \colon \tilde{\Cc} \longrightarrow \Aa^{\textnormal{sm}}$ is the smooth universal spectral curve of genus $g'$.
 
    Now, consider the relative Hitchin basis $\Aa \longrightarrow \Spec R$. The work of Hansen and Scholze \cite{Scholze} provides a relative perverse $t$-structure in $D^b_c(\Aa, \QQ_\ell)$ and a well-behaved category of perverse sheaves $Perv^{\text{ULA}}(\Aa/R)$.
    It follows from \cite[Proposition 2.2]{MS23} and \cite[Theorem 1.10]{Scholze} that a morphism 
   \begin{align} \label{morph}
     \bigoplus_{i=0}^{2g'} IC(\wedge^i R^1 \pi_* \QQ_\ell) \longrightarrow {Rh_{\GLn}}_* {IC_{M_{\GLn, R}^d}}  
   \end{align} 
    exists in $Perv^{\text{ULA}}(\Aa/R)$. Moreover, its base change to $\CC$ is an isomorphism \cite[Theorem 0.4]{MS23}. Hence, using the equivalence of categories in \cite[Lemme 6.1.9]{BBD}, there exists a dense open set $B' \longrightarrow B = \Spec R$ such that the restriction of (\ref{morph}) to $B'$ is an isomorphism. 
    Using the compatibilities between perverse $t$-structures in \cite[Theorem 1.1]{Scholze}, for any finite field $k$ with $\Spec k \rightarrow B'$, we obtain an isomorphism over $k$ and a corresponding equality of traces of Frobenius. Repeating for $d=0$ and $d=1$, we get for any $\Spec k \rightarrow B'$. 
    $$\tr (\Fr \ | \ IH_c (M_{\GLn, k}^0, \underline{\QQ}_\ell)) = \tr (\Fr \ | \ IH_c(M_{\GLn, k}^1, \underline{\QQ}_\ell)). $$

    Now, the albanese map gives a splitting \cite[(10)]{felisettiMauri} :
    \begin{align} \label{splitting}
     IH_c(M_{\GLn, k}^d )&\simeq IH_c(M_{\PGLn, k}^d) \otimes H_c (M_{GL_1, k}^d).   
    \end{align} 
    Since $M_{GL_1, k}^0 \simeq M_{GL_1, k}^1$, we get the claim. 
\end{proof}

    We now turn to the rest of the right-hand side in (\ref{degree 1 pgln coho}). Let $\gamma \neq 1 \in  \Gamma$. The $\gamma$-fixed loci have been well studied in the literature and we refer to \cite{NR, HT, MSendoscopicHTcoprime} for an introduction.
    
\begin{claim}[Remaining terms in (\ref{degree 1 pgln coho})]
    Let $\gamma \neq 1$. Then, for any line bundle $L_1 \in \Pp ic^1_{\Cc/k_F}(k_F)$,  
    \begin{align*}
    \tr (\Fr \ | \ H_c (M_{\SLn, k}^{L_1,\gamma}, \underline{\QQ}_\ell)^\Gamma ) 
    &= \tr (\Fr \ | \ H_c (M_{\SLn, k}^{0,\gamma}, \underline{\QQ}_\ell)_\kappa )
\end{align*}
where $\kappa = \omega(\gamma)$ (Remark \ref{rmk weil pairing}). \end{claim}

\begin{proof}
    Since $n$ is prime, $\gamma$ has order $n$. Let $\Cc_\gamma \xrightarrow{\pi} \Cc$ be the degree $n$ Galois cover corresponding to $\gamma$ and $G$ denote its Galois group. The $\GLn$-Higgs bundles on $\Cc$ fixed by $\gamma$ correspond to push-forwards of line bundles on $\Cc_\gamma$ with a section of $\pi^*D$. The fixed determinant Higgs bundles correspond to a fibre of the Norm map. Since these fibres are all torsors under $Nm^{-1}_{\Cc_\gamma / \Cc}(\Oo_\Cc)$, we get : 
     \begin{align*}
M_{\SLn}^{L_1,\gamma} \big/\Gamma &\simeq \Biggl( \frac{ \Nm^{-1}_{\Cc_\gamma / \Cc}(L_1') \times H^0(\Cc_\gamma, \pi^* D) }{G}  \Biggr) \bigg/ \Gamma \\
   &\simeq \Biggl( \frac{ \Nm^{-1}_{\Cc_\gamma / \Cc}(L_0') \times H^0(\Cc_\gamma, \pi^* D) }{G}  \Biggr) \bigg/ \Gamma \\
   &\simeq  M_{\SLn}^{0,\gamma} \big/\Gamma 
      \end{align*}  
where\footnote{The line bundle $\det(\pi_* \Oo_{\Cc_\gamma})$ is $2$-torsion, it is $L_\gamma^{\frac{n}{2}}$ if $n$ is even and $\Oo_\Cc$ otherwise \cite{HT}.} $L_0' = \det\pi_* \Oo_{\Cc_\gamma}$  and $L_1' = L_1 \otimes \det \pi_* \Oo_{\Cc_\gamma}$. 

However, the isomorphism $Nm^{-1}_{\Cc_\gamma / \Cc}(L_1') \simeq Nm^{-1}_{\Cc_\gamma / \Cc}(L_0')$  is not $G$-equivariant, since $G$ acts cyclically on the components of $Nm^{-1}_{\Cc_\gamma / \Cc}(L_1') $ but fixes the components of $Nm^{-1}_{\Cc_\gamma / \Cc}(L_0')$ \cite[Proposition 3.5]{NR}. Since in the latter case $\Gamma$ acts cyclically on components, the $\Gamma$-representation $H^* (M_{\SLn}^{0,\gamma})$ is regular. Hence, 
\begin{align*}
    \tr (\Fr \ | \ H_c (M_{\SLn, k}^{L_1,\gamma}, \underline{\QQ}_\ell)^\Gamma ) &= \tr (\Fr \ | \ H_c (M_{\SLn, k}^{0,\gamma}, \underline{\QQ}_\ell)^\Gamma ) \\
    &= \tr (\Fr \ | \ H_c (M_{\SLn, k}^{0,\gamma}, \underline{\QQ}_\ell)_\kappa ).
\end{align*}
\end{proof}

 \textit{Step 3.} We use \cite{MSnoncoprime}'s isomorphism (\ref{MSeq}). The morphism $$c_\kappa \colon (Rh_{\SLn,*} IC_{M_{\SLn}^{0}})_{\kappa} \longrightarrow {i_\gamma}_* (R{h_\gamma}_* IC _{M_{\SLn}^{0, \gamma}})_{\kappa} [-F_\gamma] $$ is of geometric nature, being constructed as a combination of algebraic correspondences \cite[§0.3]{MS23}. 
Hence, the morphism exists over $B = \Spec R$, and its base change to $\CC$ is an isomorphism. By spreading out, there exists an open $B'' \subset B$, such that for all $\Spec k \longrightarrow B''$, 
$$c_{\kappa, k} \colon (Rh_{\SLn,*} IC _{M_{\SLn, k}^{0}})_{\kappa} \xlongrightarrow{\sim} {i_\gamma}_* (R{h_\gamma}_* IC _{M_{\SLn, k}^{0, \gamma}})_{\kappa} [-F_\gamma] $$
is an isomorphism. 

Finally, we can consider an open set $\tilde{B}$ contained in $B'$ and $B''$. Then, for $\Spec k_F \longrightarrow \tilde{B}$: 

\begin{align*}
    \int_{M_{\SLn}^{0}(\Oo_F)^\sharp} f_{\alpha} \ \mu_{\can}  &= q^{- \dim M_{\PGLn}^0} \cdot \bigg( \tr (\Fr \ | \ IH_c(M_{\PGLn,k_F}^0, \underline{\QQ}_\ell)) + \sum_{\gamma \neq 1}  \tr (\Fr \ | \ H_c^{* -F_\gamma}(M_{\SLn,k_F}^{0,\gamma}, \underline{\QQ}_\ell)_\kappa) \bigg) \\
    &= q^{- \dim M_{\SLn}^0} \sum_{\gamma \in \Gamma} \tr (\Fr \ | \ IH_c(M_{\SLn, k_F}^{0}, \underline{\QQ}_\ell)_\kappa) \\
    &= q^{- \dim M_{\SLn}^0} \tr (\Fr \ | \ IH_c(M_{\SLn, k_F}^{0}, \underline{\QQ}_\ell)).
\end{align*}
\end{proof}

\begin{remark}
    We could weaken our assumptions by assuming that $n$ is an odd prime and $D$ arbitrary, since for odd $n$, $d' \equiv d \mod n$. Also, instead of $L = N^{d}$ we could assume that $L\otimes N^{- d}$ has a $\mathit{n}$th root. Finally, we expect that the proof extends to the case where $n$ is not prime, but it requires to control $G$-equivariance in the $\chi$-independence isomorphisms.
\end{remark}

\begin{remark}
    The same argument using our refinement (\autoref{fourier}) shows that under the same hypothesis, for $\kappa \in \hat{\Gamma}$,
    $$  \frac{1}{ |\Gamma|} \sum_{\nu \in \Gamma} \Big( \int_{M_{\SLn}^{L_\nu} (\Oo_F)^\sharp} f_{\alpha}^{e'} \ \mu_{\can} \Big) \kappa(\nu)  = q^{- \dim M_{\SLn}^0} \tr (\Fr \ | \ IH_c(M_{\SLn, k_F}^0, \underline{\QQ}_\ell)_{\kappa}).$$
\end{remark}

\subsection{Link to BPS cohomology when \texorpdfstring{$D = K$}{D=K}.} In this subsection we discuss a possible link between \autoref{mainthm} and the recent Hausel--Thaddeus type conjecture formulated in terms of BPS cohomology \cite{Davison}. It is speculative because the $BPS$ sheaf has not yet be constructed as $\ell$-adic sheaf.

Focusing on the case where $D=K_{\Cc}$, \autoref{mainthm} reads, for all $ d, e \in \ZZ/n\ZZ$,
 \begin{align} \label{p-adic D=K}
     \int_{M_{\SLn}^{L}(\Oo_F)^\sharp} f_\alpha^e  \ \mu_{\can} &= 
    \int_{M^e_{\PGLn}(\Oo_F)^\sharp} f_{\alpha_N}^d   \ \mu_{\can}
 \end{align}
 where $L \in \Pp ic^d(\Cc)$, assuming that $L\otimes N^{-d}$ has a $\mathit{n}$th root. 

In \cite{Davison}, authors construct a perverse sheaf $\phi_{BPS}$ on the good moduli space of the loop stacks $\Ll \Mm_{\SLn}^{L}$ and $\Ll \Mm_{\PGLn}^e$. The Euler characteristic specialises to invariants related to counting of $BPS$ (for Bogomol'nyi--Prasad--Sommerfield) states in physics, as explained in \cite[\textsection 1.2.14]{Davison}

They propose a Hausel--Thaddeus-type conjecture \cite[Conjecture 10.3.25]{Davison} which reads :  
\begin{equation} \label{BPSconj}
    H^*_{BPS} (\Ll M_{\SLn}^{L})_e \simeq H^*_{BPS} (\Ll M_{\PGLn}^{e})_d
\end{equation} 
where $\Ll M_{\SLn}^{L}$ denote the good moduli space of the loop stack $\Ll \Mm_{\SLn}^{L}$ and similarly for $\PGLn$. 

The $BPS$ cohomology of $\Ll M_{\SLn}^{L}$ inherits a $Z(\SLn) = \mu_n$ action. Interpreting $e$ as a character of $\mu_n$, $H^*_{BPS} (\Ll M_{\SLn}^{L})$ decomposes into a sum of isotypical components denoted $H^*_{BPS} (\Ll M_{\SLn}^{L})_e$. A similar decomposition is constructed on the $\PGLn$ side in \cite[Subsection 10.3.24]{Davison}. 

Motivated by the similarities between (\ref{p-adic D=K}) and (\ref{BPSconj}), we formulate the following conjecture. 

\begin{conjecture} \label{conjecture} Let $S = \Spec \Oo_F$ where $\Oo_F$ is a local ring of finite residue field $k_F$. Let $D=K_\Cc$ and $L \in \Pp ic^d(\Cc)$ for $d \in \ZZ$. Let $\alpha$ denote the gerbe $\Mm^{L}_{\widetilde{\SLn}} \longrightarrow \Mm^{L,\rig}_{\SLn}$. Let $q = |k_F|$. Using the notations introduced in Notation \ref{notation frob},  
    $$ \int_{M_{\SLn}^{L}(\Oo_F)^\sharp} f_\alpha  \ \mu_{\can} =  q^{- \dim M_{\SLn}^{L}} \tr( \Fr \ | \  H^*_{BPS} (\Ll M_{\SLn, k_F}^{L})_1  ).$$ 
\end{conjecture}

Beyond the coprime case, Proposition \ref{prop link IH} gives some evidence for this conjecture, even though we assumed $D> K_\Cc$ there. Indeed, in that situation, the stack  $\Mm_{\SLn}^{L}$ is smooth and in this case intersection and $BPS$ cohomologies coincide \cite{Davison}. Moreover, in the context of $\GLn$-Higgs bundles, $p$-adic integrals are related to $BPS$ invariants (see \cite{COW} and \cite[Lemma 6.2]{GWZ24}). In \cite{COW}, the relation only holds for a given power of the gerbe function which corresponds to the genus $g'$ of the spectral curves, which is congruent to $e'$ modulo $n$ precisely when $e=1$.






\appendix 

\section{Comparison of gerbes}
\label{appendix}


Let $S = \Spec k$ for a field $k$. As before, pick a curve $\Cc$ over $S$, line bundles $L, D$ over $\Cc$ and consider the coarse moduli space of determinant $L$, traceless semi-stable Higgs bundles $M^L_{\SLn}$. There are two ways to describe $M^L_{\SLn}$ as a good moduli space of a stack endowed with a $\mu_n$-gerbe. As before, we consider gerbes with respect to the étale topology. 

\begin{Construction}
    This construction was used in \cite{GWZ} and the present work. Recall the stack $\Mm_{\SLn}^L$ whose objects are :
    \begin{align*}
    \Ob(\Mm_{\SLn}^L(T)) = \Bigg\{ (E, \phi) \ &| \ E \in \Vv ec_n(\Cc)(T), \  \phi \in H^0(\Cc, \mathfrak{sl}_n(E)\otimes D), \\
                                    & \ \text{semi-stable}, \ \det(E) \simeq q^*L \in \Pp ic_{C \times_S T/S}^d  (S) \Bigg\} 
\end{align*} 
where $q \colon T \times_S \Cc \longrightarrow \Cc$. The gerbe $\alpha_n$ is given by the $\mu_n$-rigidification defined in \cite{ACV}, 
$$\Mm_{\SLn}^L \xrightarrow{\alpha_n} \Mm_{\SLn}^{L,\rig},$$
which is obtained by removing the scalar action from the automorphisms.
\end{Construction}

\begin{Construction}
This construction appears in \cite{HT}. Note that in the coprime case, $\Mm_{\SLn}^{L, \rig} \simeq M^L_{\SLn}$. Fix a basepoint $\varepsilon \in \Cc(S)$. There is a universal projective bundle $\PP\EE$ over $\Cc \times \Mm_{\SLn}^{L, \rig}$. Take  $U \longrightarrow \Mm_{\SLn}^{L,\rig}$. Then, let $\alpha_{HT}$ be the stack defined by :
$$ \Ob ( \alpha_{HT}(U)) = \big\{  \EE^{\varepsilon}_{U} \in \Bb un_{\SLn}(U) \ | \ \PP(\EE^{\varepsilon}_{U}) \simeq \varepsilon^*\PP\EE_{|U} \big\}.$$ 
We call $\EE^{\varepsilon}_{U}$ a lift over $U$. Isomorphisms between lifts over $U$ are given by $\mu_n$-torsors over $U$. The étale-local existence of lifts means that $\alpha_{HT}$ is an étale $\mu_n$-gerbe. 
\end{Construction}
    
\begin{remark}
    Using the seesaw principle, it is equivalent to rephrase the condition $\det(E) \simeq q^*L$ in $\Ob(\Mm_{\SLn}^L(T))$ as $\forall \ t \in |T|, \  \det(t^*E) \simeq L \in \Pp ic^d  (\Cc) \text{ and } \det(\varepsilon^*E) \simeq \Oo_T $. The latter is used in \cite{HT}. 
\end{remark}

\begin{lemma}
Both gerbes are isomorphic, i.e. there is a commuting diagram :
$$\begin{tikzcd}
    \Mm_{\SLn}^L \arrow[dr, "\alpha_n"] \arrow[rr, "\sim"] & & \alpha_{HT} \arrow[dl] \\
    &\Mm_{\SLn}^{L,\rig} &
\end{tikzcd}$$
\end{lemma}

\begin{proof}
A section of $\alpha_n$ over $U \longrightarrow \Mm_{\SLn}^{L,\rig}$ gives a universal (Higgs) bundle $\EE_U$ of determinant $q^*L \in \Pp ic(\Cc \times U )$. Then, $\det(\varepsilon^* \EE_U) \simeq \Oo_U$, and the universality implies that the projectivisation of $\EE_U$ is $\PP\EE_{|U}$. Hence, $\varepsilon^* \EE_U$ provides a section of $\alpha_{HT}$ over $U$. We just constructed a map of $\mu_n$-gerbes $\alpha_n \longrightarrow \alpha_{HT}$. As for torsors, a map between $\mu_n$-gerbes is necessarily an isomorphism. 
\end{proof}




\AtNextBibliography{\footnotesize}
\printbibliography
\end{document}